\documentclass[10pt]{article}
\title{\vspace{-1.0cm}\textbf{Scaling limits of anisotropic growth on logarithmic time-scales}}
\author{\textbf{George Liddle\footnote{Department of Mathematics and Statistics, Lancaster University, United Kingdom. Email: georgeliddle95@gmail.com}, Amanda Turner}\footnote{School of Mathematics, Leeds University, United Kingdom. Email: a.turner5@leeds.ac.uk. Supported by EPSRC grant EP/T027940/1}}

\usepackage{geometry}
\newgeometry{vmargin={1in}, hmargin={1in}}
\usepackage{amsthm}
\usepackage{amsmath}
\usepackage{color}
\usepackage{ dsfont }
\usepackage{graphicx}
\usepackage{bbold}
\usepackage{xcolor}
\usepackage{eufrak}
\usepackage{comment}
\usepackage{bbm}

\newtheorem{theorem}{Theorem}[section]
\newtheorem*{theorem*}{Theorem}
\newtheorem{lemma}[theorem]{Lemma}

\newtheorem{corollary}[theorem]{Corollary}
\newtheorem*{corollary*}{Corollary}

\definecolor{blue}{rgb}{0.0,0.02,0.67}
\def\bb{\begin{color}{blue}}
\def\bw{\begin{color}{white}}
\def\bg{\begin{color}{green}}
\def\br{\begin{color}{red}}
\def\bbr{\begin{color}{brown}}
\def\eg{\end{color}}
\def\ew{\end{color}}
\def\er{\end{color}}
\def\eb{\end{color}}

\date{}
\usepackage[utf8]{inputenc}

\usepackage{amssymb}

\usepackage[intoc]{nomencl}
\makenomenclature

\usepackage{etoolbox}
\renewcommand\nomgroup[1]{%
  \item[\bfseries
  \ifstrequal{#1}{A}{Conformal random growth models}{%
  \ifstrequal{#1}{B}{Anistotropic Hastings-Levitov model AHL$(\nu)$}{%
  \ifstrequal{#1}{C}{Other Symbols}{}}}%
]}

\begin{document}
\maketitle
\begin{abstract}
We study the anisotropic version of the Hastings-Levitov model AHL$(\nu)$. Previous results have shown that on bounded time-scales the harmonic measure on the boundary of the cluster converges, in the small-particle limit, to the solution of a deterministic ordinary differential equation. We consider the evolution of the harmonic measure on time-scales which grow logarithmically as the particle size converges to zero and show that, over this time-scale, the leading order behaviour of the harmonic measure becomes random. Specifically, we show that there exists a critical logarithmic time window in which the harmonic measure flow, started from the unstable fixed point, moves stochastically from the unstable point towards a stable fixed point, and we show that the full trajectory can be characterised in terms of a single Gaussian random variable. 
\end{abstract}

\textbf{Key words}: Planar random growth; scaling limits; fluctuations; anisotropic growth.


\noindent
\section{Introduction}

The aim of this paper is to study the behaviour of a class of random growth processes modelled using conformal mappings. In recent years, many models have been introduced in order to study various real world random growth processes from lightning strikes and mineral aggregation to tumoral growth. The most well known examples include the Eden model \cite{E} and DLA \cite{DLA}. These models, built on a lattice, have been well studied but rigorous results have proved difficult to come by (see for example \cite{HK2}). One reason for this is that lattice based models provide little in the way of mathematical techniques that can be used to study their behaviour. One way to combat this difficulty is to form off-lattice versions of the models using conformal mappings which allows us to study the processes in the complex plane and use complex analysis techniques. 

The models are constructed using a basic particle shape. In this paper, we will restrict ourselves to `slit' particles, although the results can be extended to more general particle shapes. Given a desired slit length $d > 0$, set 
$$c=\log \left (1+\frac{d^2}{4(1+d)} \right ) > 0.$$
Then, setting $\Delta = \{ |z|>1\}$ to be the exterior unit disk, there exists a unique single slit mapping $$f_{(c)}:\Delta \to \Delta \backslash (1,1+d]$$ which takes the exterior of the unit disk to itself minus a slit of length $d$ at $z=1$.
It can be shown (see for example \cite{STV2}), that
$$f_{(c)}(z)=e^{c}z+\mathcal{O}(1)$$
as $|z| \to \infty$. We refer to $c$ as the capacity of the slit. It follows that there is a one-to-one correspondence between capacities and conformal maps attaching a slit onto the boundary of the disk at $1$. Given a sequence  of capacities, $c_n \in (0, \infty)$, and attachment angles, $\theta_n \in [0,1]$ (identifying the unit circle with the interval $[0,1]$), we define a sequence of rescaled and rotated slit mappings by setting 
\begin{equation}\label{fmaps}
f_n(z)=e^{2\pi i\theta_n} f_{(c_n)}(ze^{-2\pi i  \theta_n}).
\end{equation}
This mapping corresponds to attaching a slit of capacity $c_n$ to the unit circle at position $e^{2 \pi i \theta_n}$.
One can then form a growing sequence of clusters by composing the slit maps. Let $\phi_0(z) = z$ so that $\phi_0 : \Delta \to \mathbb{C} \setminus K_0$ is a bi-holomorphic map where $K_0= \{|z|\leq 1\}$. Now define
$$\phi_{n+1}=\phi_n \circ f_{n+1}=f_{1}\circ f_{2} \circ \cdots \circ f_{n+1}.$$
Then there exists a sequence of compact sets $K_n$ satisfying $K_{n} \subsetneq K_{n+1}$ for which $\phi_n: \Delta \to \mathbb{C} \setminus K_n$ is a bi-holomorphic map which fixes the point at infinity. The sequences of capacities $c_n$ and angles $\theta_n$ therefore define a growing sequence of clusters $K_n$. By choosing the attachment angles and capacities appropriately we can model a wide class of growth processes.

\subsection{Previous work}

The Hastings-Levitov model \cite{HL} is formed using conformal maps as described above when the attachment angles are chosen uniformly. This choice represents a good model for many of the real world processes where particles diffuse onto the boundary at each iteration (for a more detailed description see, for example, \cite{GL}). Furthermore, the capacities are chosen as,
$$c_n=c|\phi_{n-1} '(e^{i\theta_n})|^{-\alpha}.$$
The parameter $\alpha$ allows us to vary between off-lattice versions of the previously well studied models by varying the size of the attached slits. By choosing the capacities and attachment angles in this way we can model a wide class of real world Laplacian growth processes where the local growth rate is chosen according to harmonic measure. In recent years research into the Hastings-Levitov model has been fruitful. The majority of the results have concentrated on the scaling limits of the model in the small-particle limit where the cluster $\phi_n$ is analysed as the particle capacity $c\to 0$ while $n\to \infty$ with $nc\sim t$ for some $t$. In \cite{NT} Norris and Turner show that in the small-particle limit, for $\alpha=0$, the limiting cluster behaves like a growing disk. Furthermore, in \cite{STV1} Sola, Turner and Viklund show that in the small-particle limit the shape of the cluster approaches a disk for all $\alpha\geq 0$ in the presence of a sufficiently strong regularisation. Moreover, Silvestri \cite{S} shows that the fluctuations on the boundary, for HL$(0)$, in the small-particle limit can be characterised by a log-correlated Gaussian field.

The Hastings-Levitov model has also been evaluated under another scaling limit where rather than letting $c\to0$ as $n\to\infty$, instead, the limit of the cluster is found by rescaling the whole cluster by the logarithmic capacity of the cluster at time $n$, before taking limits as the number of particles tends to infinity.  In \cite{RZ} Rohde and Zinsmeister introduce a regularisation to the model and show that in the case of $\alpha=0$ the rescaled cluster converges to a (random) limit with respect to the topology of normalised exterior Riemann maps. In \cite{GL}, Liddle and Turner show that for $\alpha=0$ the scaling limit of the cluster under capacity rescaling is not a disk. Furthermore the authors study a regularised version of the model and show that for $0<\alpha<2$ the scaling limit under capacity rescaling is a disk and the fluctuations behave like a Gaussian field.

However, studies have also been made into a wider class of processes where the particles are not attached uniformly. The ALE$(\alpha, \eta)$ model introduced in \cite{STV2} generalises the Hastings-Levitov by choosing the local growth rate to be determined by $|\phi_n'|^{-\eta}$ where $\eta \in \mathbb{R}$. The authors show that there exists a phase transition at $\eta=1$ when $\alpha=0$ where the limiting shape in the small particle limit transitions from a disk to a radial slit. This model is also studied in \cite{NST} and \cite{NVT}, where it is shown that clusters are disk-like provided $\alpha+\eta \leq 1$, and there is a phase-transition present in the fluctuations when $\alpha+\eta=1$. In \cite{FH}, Higgs considers the model for $\alpha=0$ and for large negative values of the parameter $\eta$ where the particles are attached in areas of low harmonic measure and shows that there exists a phase transition beyond which the ALE cluster converges to an SLE$_4$ curve.

The final generalisation is the anisotropic version of the Hastings-Levitov model AHL$(\nu)$ which will be the subject of this paper. The AHL$(\nu)$ model is a variation of the Hastings-Levitov model HL$(0)$. The model is constructed as above with the capacities chosen to be a fixed value $c$ and the attachment angles chosen to be i.i.d on the unit circle but, in contrast to HL$(0)$ where the attachment points are distributed uniformly, they are chosen  according to some non-uniform probability measure $\nu$. In \cite{JTS} Sola, Turner and Viklund show that if 
 $\overline{\phi}$ is the solution to Loewner-Kufarev equation driven by the measure $\nu$ and  $\phi_n=f_1\circ \cdots \circ f_n$ then $\phi_{\lfloor t/c \rfloor}\to\overline{\phi}$ uniformly on compact sets in the exterior unit disk almost surely as $c\to 0$. Furthermore, the authors study the scaling limits of the harmonic measure flow (defined below) and show that on bounded time-scales they can be described by the solution to a deterministic ordinary differential equation related to the Loewner equation. In contrast, in this paper we will evaluate the scaling limits of the harmonic measure flows on logarithmic time-scales and we will show that the leading order behaviour is random. 

\subsection{Overview of the main results and physical interpretation}

In this paper, we study the anisotropic Hastings-Levitov model introduced in \cite{JTS} as AHL$(\nu)$.  The shape of the cluster in the small particle limit is described in \cite{JTS}. However, we often want to understand the ancestral path of each attached particle. Evaluating how the harmonic measure evolves on the boundary of the cluster allows us to do so. For the purpose of the introduction we define the discrete harmonic measure flow for $x\in \mathbb{R}$ as
\begin{align*}
X_n(x)=\frac{1}{2\pi i}\log (\Gamma_n(e^{2\pi i x})).
\end{align*}
where $\Gamma_n(z)=\phi_{n}^{-1}(z)=f_n^{-1}\circ \cdots \circ f_1^{-1}(z)$. However, some care is needed in how to interpret this definition. Although $\phi_n(z)$ can be defined when $|z|=1$ by requiring that the extended map $\phi_n: \overline{\Delta} \to \overline{\mathbb{C} \setminus K_n}$ is continuous, this map is no longer bijective. Therefore $\Gamma_n(z)$ is not well-defined everywhere on the cluster boundary. We explain how to overcome this difficulty in Section \ref{sec:harm} where we give a more explicit definition of what we mean by harmonic measure flow. The function $X_n(x)$ tells us how the harmonic measure on the cluster boundary, as seen from infinity, evolves under the map $\phi_n$. Specifically, if $x < y < 1 + x$, then the quantity $X_n(y) - X_n(x)$ gives the harmonic measure of the portion of the cluster boundary between $e^{2 \pi i x}$ and $e^{2 \pi i y}$,  traced out in an anticlockwise direction. Our aim is to identify how this function evolves in the scaling limit as the particle size becomes small, while the number of particles increases. It is convenient to embed harmonic measure flow into continuous time so we will consider the scaling limit of $X_{n(t)}(x)$ on logarithmic time-scales as $c\to 0$ where $n(t)=\left\lfloor \frac{t}{c} \right\rfloor $.

In this section we describe the main results of the paper and their physical interpretation. As an illustration we consider an AHL($\nu$) cluster where the measure $\nu$ is concentrated on a segment of the disk such as  $d \nu( e^{2\pi i x})= 2\mathbbm{1}(x \in [0, 1/2]) dx$, as chosen in Figure 2 from \cite{JTS} which has been reproduced in this paper as Figures \ref{fig3} and \ref{fig4}. Figure \ref{fig3} shows an AHL$(\nu)$ cluster together with its limit shape, and Figure \ref{fig4} provides the corresponding evolution of harmonic measure on the boundary of the cluster together with the solution to a deterministic ODE. This figure illustrates how, on compact time intervals, the harmonic measure flow on the boundary of the cluster converges to the solution of the deterministic ODE. We also include cartoons (Figures \ref{fig2}, \ref{fig6} and \ref{fig1} below) to aid our descriptions. It should be noted that these cartoons are not accurate simulations and are not drawn to scale but instead serve to highlight specific features of a potential evolution of an AHL$(\nu)$ cluster. 

\begin{figure}[!htb]
\centering
 \includegraphics[width=0.75\linewidth]{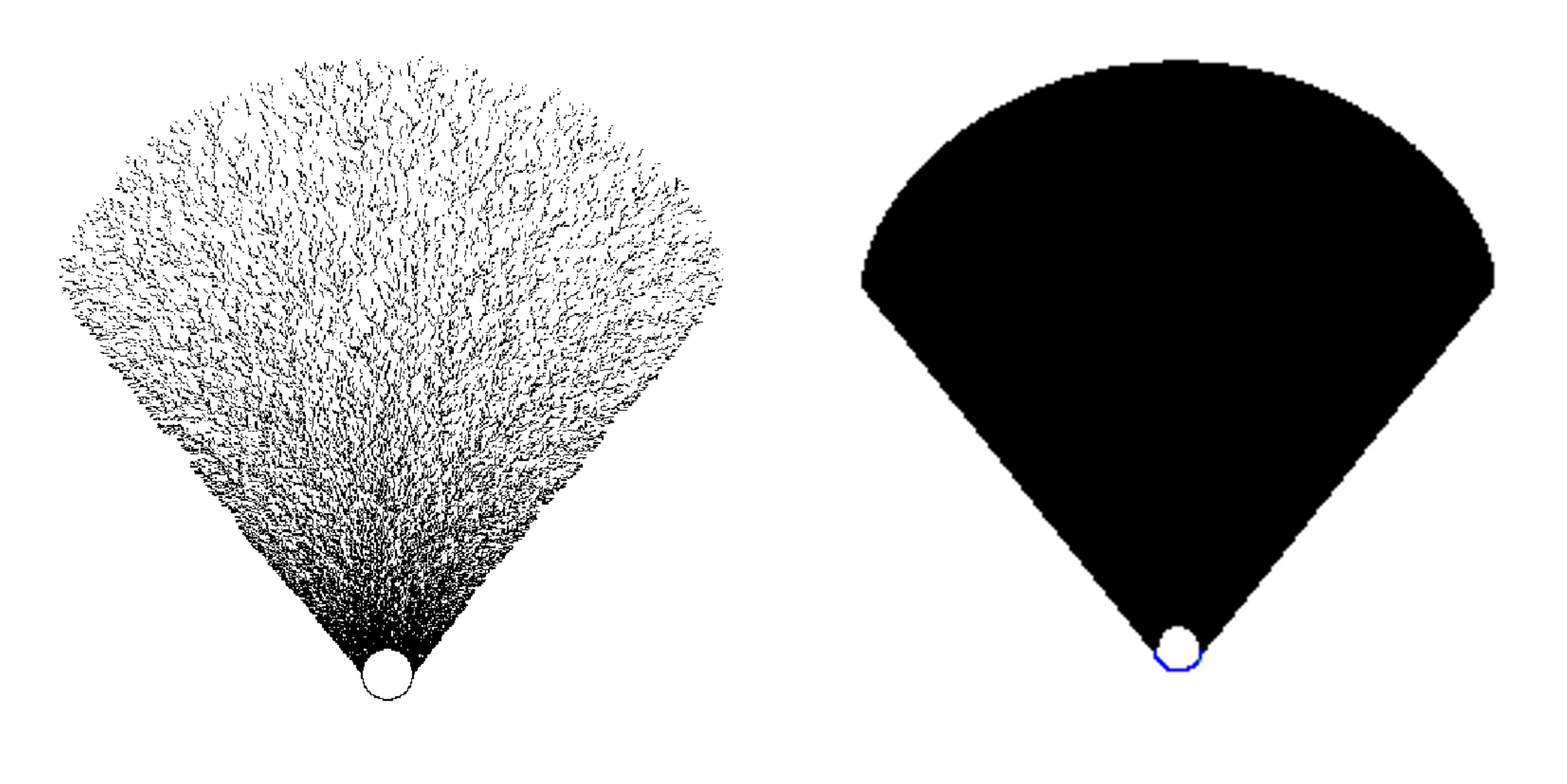}
  \caption{ An example of a AHL$(\nu)$ cluster (left) and the corresponding Loewner hull (right) from \cite{JTS}.}\label{fig3}
\end{figure}
\begin{figure}[!htb]
\centering
 \includegraphics[width=0.9\linewidth]{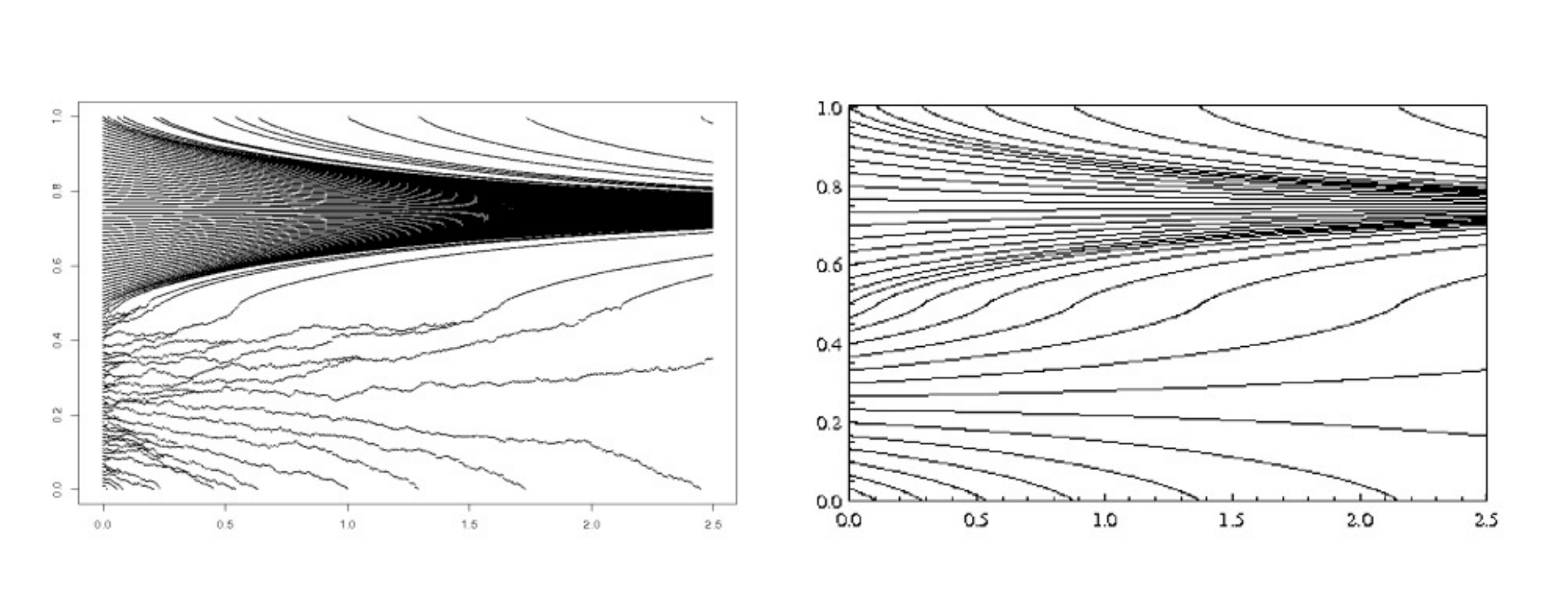}
  \caption{Harmonic measure flow $X_{n(t)}(x)$ on the boundary of AHL$(\nu)$ plotted against time with the departure point $x$ indicated on the $y$-axis (left) and the solution to a corresponding deterministic ODE (right) from \cite{JTS}.}\label{fig4}
\end{figure}

In order to state our main results we need to introduce some preliminary assumptions and notation. We shall assume that the measure $\nu$ has a twice continuously differentiable density function $h_\nu$, defined on $[0,1]$ and extended periodically to $\mathbb{R}$, so
\[
\nu(e^{2 \pi i x}) = h_\nu(x) dx.
\]
In practice, this assumption is not very restrictive as most natural measures in this setting can be arbitrarily well approximated by a measure satisfying this condition. Define the function $b: \mathbb{R} \to \mathbb{R}$ to be the Hilbert transform of the measure $\nu$,
\begin{equation} \label{bequation}
b(x)=\frac{1}{2\pi}\int_0^1 \cot (\pi z)( h_{\nu}(x-z)-h_{\nu}(x))dz.
\end{equation}
For $t \in [0, \infty)$ and $x \in \mathbb{R}$, let $\psi_t(x)$ be the solution to the ordinary differential equation,
\begin{equation}\label{ode1}
\dot{\psi}_t(x)=b(\psi_t(x))
\end{equation}
with initial condition $\psi_0(x)=x$.

We consider the evolution of the harmonic measure $X_{n(t)}$. Our first main result, appearing later as Theorem \ref{3.6II}, extends the result in \cite{JTS} on convergence of the harmonic measure flow from compact time intervals to logarithmic time-scales.

\begin{theorem*}
Set $T_0=\frac{1}{4\|b'\|_{\infty}}\left(\log(c^{-1})-3\log(\log(c^{-1}))\right)$. Then, for every $x \in \mathbb{R}$,
\[
\sup_{0 \leq t \leq T_0} |X_{n(t)}(x) - \psi_t(x)| \to 0
\]
in probability as $c\to 0$.
\end{theorem*}
\begin{figure}[!htb]
\centering
 \includegraphics[width=0.8\linewidth]{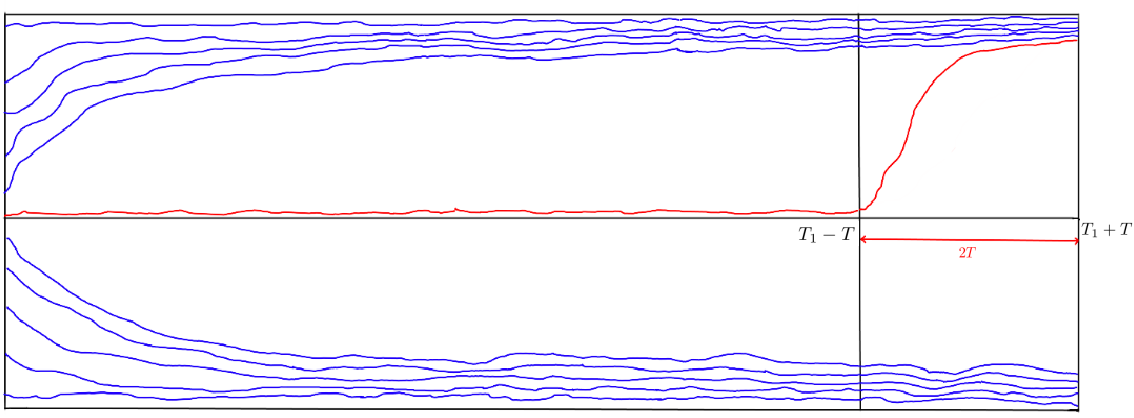}
  \caption{Cartoon illustrating the evolution of $X_{n(t)}$ on logarithmic-timescales. }\label{fig2}
  \end{figure}

This result shows that, on the given logarithmic time-scale, all the trajectories of the harmonic measure  process $X_{n(t)}(x)$ remain close to the corresponding deterministic trajectories of the ordinary differential equation $\psi_t(x)$. We illustrate this in Figure \ref{fig2} with each blue trajectory remaining close to the solution to the ODE up to this time. The value of $T_0$ is not shown explicitly in this figure, but is smaller than the marked value of $T_1 - T$.

Yet, consider the simulations in Figure \ref{fig4} taken from \cite{JTS}. In this figure, one can observe that the harmonic measure started at the unstable fixed point, $x=1/4$, of the deterministic ODE begins to deflect away from the value of fixed point towards the end of the time interval depicted in the simulation. This behaviour is captured in the following result which appears later as Corollary \ref{corol4.8}. 

\begin{theorem*}
Let $a_u$ be an unstable fixed point of $\psi_t(x)$. Then for $0<t<\infty$, $X_{n(t)}(a_u)$ converges to $\psi_{t}\left(a_u+c^{\frac{1}{4}}Z_{\infty}(a_u)\right)$ in probability as 
$c\to 0$, where $ Z_{\infty}(a_u)$ is a Gaussian random variable with mean $0$ and variance which can be given explicitly in terms of the measure $\nu$. 
\end{theorem*}

This result tells us that the harmonic measure flow started from the unstable point $a_u$ tracks the solution to the ODE started from some random perturbation of $a_u$. This perturbation is amplified by the ODE system over logarithmic time-scales. Therefore there exists a random time $T_1$, on a logarithmic time-scale, by which point $X_{n(t)}(a_u)$ has moved a macroscopic distance away from the fixed point $a_u$. Once the process reaches this macroscopic distance it remains close to the trajectory of the ODE started from that distance at time $T_1$. However, we know that for 1-dimensional ODEs, trajectories started away from unstable points converge to stable points. Therefore, once the process gets close enough to the stable point it remains close.  
This behaviour is illustrated in Figure \ref{fig2}.  The red trajectory represents the behaviour of the harmonic measure started at the unstable point. If it converged to the solution of the ODE we would expect this trajectory to remain close to the unstable point for all time. However, the cartoon shows the stochastic nature of the path the trajectory takes from the unstable point towards a stable point during the critical time window.

We now consider what the physical interpretation of this result is on the AHL$(\nu)$ cluster itself.  We will describe this using the notion of gap paths. The explicit definition of gap paths is provided in \cite{NT}, however, intuitively the gap path from a point $z\in\mathbb{C}$ represents the shortest path from $z$ to outside the boundary of the cluster. This is demonstrated in Figure \ref{fig6} with particles represented as disks. We consider the point $z$ and imagine a piece of string attached at $z$ and pulled tight vertically until we leave the boundary of cluster. This represents the gap path of the point $z$ and is indicated by the red line in Figure \ref{fig6}. Note that the gap paths are dependent on the number of particles $n$ attached to the cluster. The gap paths cannot intersect the particles unless $z$ initially is contained inside a particle in which case we choose the shortest path to leave the particle we are contained in and then proceed as above. It is shown in \cite{NT} that in the limit as $c\to 0$ the trajectories of the gap paths are described by the harmonic measure flow, under a deterministic transformation.

\begin{figure}[!htb]
\minipage{0.45\textwidth}
  \includegraphics[width=\linewidth]{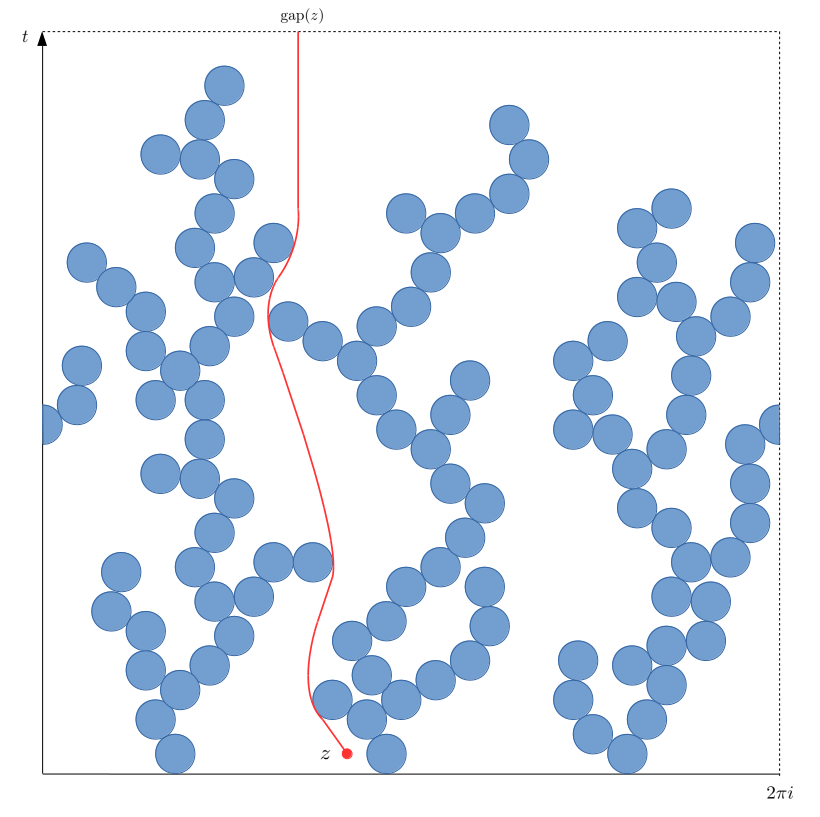}
  \caption{An example of a gap path}\label{fig6}
\endminipage\hfill
\minipage{0.45\textwidth}
  \includegraphics[width=\linewidth]{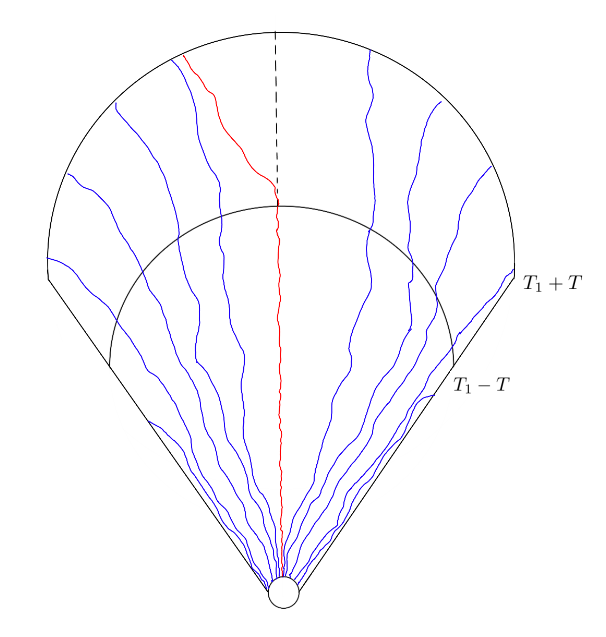}
\caption{An example of a possible AHL$(\nu)$ \newline cluster}\label{fig1}
\endminipage
 \begin{center}
 
 \end{center}
\end{figure}

With the notion of gap paths in mind we can describe how our result can be used to obtain a physical interpretation the behaviour of the cluster on longer time-scales. We use the cartoon in Figure \ref{fig1} to illustrate the physical interpretation obtained by applying our result to the example from Figure \ref{fig3}.
The harmonic measure flow allows us to map the ancestry of each the particles on the boundary of the growing cluster to an origin on the boundary of the unit disk. In Figure \ref{fig1}, the unstable point is at the centre of the arc on the unit disk and the stable points on either edge. Consider the gap path of a point near the origin. On compact time intervals we expect particles attached away from the stable points to have ancestors attached near the unstable point and thus, as the gap path cannot intersect the particles, we would expect the gap path of a particle near the origin to be vertical. However, as we enter the critical time window the harmonic measure flow is no longer close to solution of the ODE started from the unstable point but instead follows a trajectory started at a macroscopic distance from the unstable point. Therefore, the successive particles are not attached vertically and the gap path becomes asymmetric as indicated by the red path in Figure \ref{fig1}. The direction the gap path follows is  dependent on the sign of $Z_{\infty}(a_u)$. Therefore, a physical consequence of our results applied to this example is that on bounded time-scales the ancestral process remains symmetric around the vertical axis whereas, as we enter the critical time window, the process becomes asymmetric.

The outline of the paper is as follows. In Section \ref{sec:harm} we define the harmonic measure flow and provide estimates that will be used in the remainder of the paper. In Section \ref{sec:logconv} we show that the harmonic measure flow $X_{n(t)}$ remains close to a deterministic ODE up to a logarithmic time. In Section \ref{sec:fluc} we characterise the distribution of the fluctuations. In Section \ref{sec:crit} we prove the existence of a critical logarithmic time window and show that within this interval the harmonic measure flow, started from the unstable point follows a stochastic path away from the unstable trajectory and towards a stable trajectory. Finally, in Section \ref{sec:convtraj}, we prove the main result about the behaviour of the entire trajectory of the harmonic measure flow starting from the unstable point. The original paper \cite{JTS} developed several techniques which have subsequently been successfully used to analyse more physical random growth models. We anticipate that the techniques developed in this paper, and in particular in Section \ref{sec:crit}, will be similarly useful in the study of random growth models which exhibit anisotropic behaviour, as well as more general situations in which one wishes to study long-time behaviour of stochastic processes.


\section{Harmonic measure flow}
\label{sec:harm}

The main aim of this section is to give a precise definition of the harmonic measure flow, and state some estimates which will be used in the remainder of the paper. Throughout, the model under consideration is the AHL$(\nu)$ model and the notation used is as defined in the introduction.

For a point $x\in (-1/2,1/2]$, set 
$$\gamma(x)=\frac{1}{2\pi i}\log \left (f_{(c)}^{-1}(e^{2\pi i x}) \right),$$
choosing the branch of logarithm which results in $x=\frac{1}{2}$ being fixed. As $f_{(c)}(z)$ is not one-to-one when $|z|=1$, we treat the two sides of the slit as distinct, when interpreting the inverse.
Explicitly, 
$$\gamma(x)=\frac{\mbox{sign}(x)}{\pi}\tan^{-1}\left(\sqrt{e^{c}\tan^2(\pi x)+e^{c}-1}\right).$$
We extend this definition to the real line as follows, if $x=k+a$ where $a\in (-1/2,1/2]$ then define $\gamma(x)=k+\gamma(a)$. Let $\tilde{\gamma}(x)=\gamma(x)-x$ and observe that this is a periodic function with period 1.

Given the sequence $\theta_n$ of i.i.d.~random variables with distribution $\nu$, set 
$$\gamma_n(x)=\gamma_n(x)=\gamma(x-\theta_n)+\theta_n.$$
It immediately follows that $\gamma_n(x)=\frac{1}{2\pi i}\log (f_n^{-1}(e^{2\pi i x}))$, interpreting the inverse and branch of logarithm appropriately. This function describes the change in angle of a point $x$ on the boundary under the transformation $f_n(x)$ and thus $\gamma_n(x)$ tells us how the harmonic measure evolves under the map $f_n(x)$ in the sense described in the introduction. Now define the discrete harmonic measure flow under the map $\phi_n$ for $x\in \mathbb{R}$ recursively by $$X_n(x)=\gamma_n(X_{n-1}(x))$$  
with $X_0(x)=x$. 
Thus if $\Gamma_n(x)=\phi_{n}^{-1}(x)=f_n^{-1}\circ \cdots \circ f_1^{-1}(x)$ then by straightforward induction 
\begin{align*}
X_n(x)
&=\frac{1}{2\pi i}\log \left (\Gamma_n(e^{2\pi i x}) \right ),
\end{align*}
as defined in the introduction.
Note that we define $X_n(x)$ in the recursive way above to ensure the correct interpretation of $\Gamma_n$ on the boundary of the cluster. Observe that 
\begin{equation*} 
X_n(x)=X_{n-1}(x)+\tilde{\gamma}(X_{n-1}(x)-\theta_n).
\end{equation*}
We can then rewrite the harmonic measure flow as
\begin{equation*} X_n(x)=x+\sum_{i=1}^n\tilde{\gamma}(X_{i-1}(x)-\theta_i).
\end{equation*}

Throughout the remainder of the paper we will rely on the following estimates taken from \cite{JTS}. In the case of slit particles, the proof can be done by elementary computation, so is not replicated here.

\begin{lemma}\label{2.2}
For $\tilde{\gamma}$ defined as above, $$\int_0^1\tilde{\gamma}(z) dz=0$$  and there exists a constant $\rho_0$ such that 
$$c^{-\frac{3}{2}}\int_0^1\tilde{\gamma}(z)^2 dz\to \rho_0 $$ 
as $c\to0$. Furthermore, there exists a constant $\delta>0$ such that the following estimates hold,
$$\|\tilde{\gamma}\|_{\infty}\leq \delta \sqrt{c},$$
and
\begin{align*}
\left|\int_0^1\tilde{\gamma}(x-\theta)^2 h_{\nu}(\theta)d\theta-\rho_0c^{\frac{3}{2}}h_{\nu}(x)\right| &\leq \delta\|h_{\nu}'\|_{\infty}c^2\log (c^{-1})
\end{align*}
for $c$ sufficiently small.
\end{lemma}

\section{Convergence on logarithmic time-scales}
\label{sec:logconv}

Recall the definition of $\psi_t(x)$ from \eqref{ode1} where $b$ is defined in \eqref{bequation}.  In this section we show that the harmonic measure flow $X_{n(t)}(x)$ is uniformly close to $\psi_t(x)$ as $c \to 0$ on time-intervals which grow logarithmically in $c$. This extends a result in \cite{JTS} which shows that this holds over compact time-intervals.

During calculations, for simplicity purposes, we will often treat $n(t)c$ as $t$. As the true difference is of order $c$ and we take the limit as $c\to0$, our results will be unchanged. 

Define the drift function
$$\beta_{\nu}(x)=\int_0^1 \tilde{\gamma}(x-z)h_{\nu}(z)dz.$$ 
Now 
\[
\mathbb{E}\left(\tilde{\gamma}(X_{i-1}(x)-\theta_i)|\mathcal{F}_{i-1}\right) = \beta_\nu(X_{i-1}(x))
\]
where $\mathcal{F}_{n}$ is the $\sigma$-algebra generated by the set $\{\theta_i \; : 1\leq i\leq n \}$. 

Set $$Y_n(x)=\tilde{\gamma}(X_{n-1}(x)-\theta_n)-\beta_{\nu}(X_{n-1}(x)).$$
Then $S_n(x)=\sum_{i=1}^n Y_i(x)$ is a martingale with respect to $\mathcal{F}_{n}$. We can write  
\begin{equation}\label{Xform}
X_n(x)=x+S_n(x)+\sum_{i=1}^n\beta_{\nu}(X_{i-1}(x)),
\end{equation}
so 
\begin{align*}
X_{n(t)}(x)-\psi_t(x)=&S_{n(t)}(x)+\sum_{i=1}^{n(t)}\beta_{\nu}(X_{i-1}(x))-\int_{0}^{t}b(\psi_s(x))ds\\
= & S_{n(t)}(x) + \sum_{i=1}^{n(t)} \left ( \beta_{\nu}(X_{i-1}(x)) - cb(X_{i-1}(x)) \right ) + c \sum_{i=1}^{n(t)} b(X_{i-1}(x)) -\int_{0}^{t}b(\psi_s(x))ds .
\end{align*}

The proof of Proposition 2 in \cite{JTS} provides the following bound.
\begin{lemma}\label{2.4}
For each $x$ and $c<\frac{1}{2}$, there exists a constant $\delta>0$ such that,
$$\left|\beta_{\nu}(x)- c b(x)\right|<\delta c^{\frac{3}{2}}\log(c^{-1}).$$
Therefore
there exists a constant $\delta>0$ such that 
\begin{align*}
|\beta_{\nu}(x)|\leq \delta c
\end{align*}
and for each $n$,
$$|Y_n(x)|\leq \delta\sqrt{c}.$$
\end{lemma}

Throughout the remainder of the paper we will assume $0<c<\frac{1}{2}$. As $h_{\nu}$ is twice continuously differentiable, by properties of Hilbert transforms it follows that $b(x)$ is also twice continuously differentiable. Furthermore, as $h_\nu$ is not the uniform measure, $b$ is not constant or equivalently $b'$ is not identically zero.
It immediately follows that
\[
\left | X_{n(t)}(x)-\psi_t(x) \right |
\leq |S_{n(t)}(x)| + \delta\log(c^{-1})  c^{\frac{3}{2}} n(t)+\|b'\|_{\infty} \int_{0}^{t} |X_{n(r)}(x)-\psi_r(x)|dr.
\]

In order to bound $|S_n(t)|$, we apply the following martingale inequality from \cite{F}. 
\begin{theorem}\label{freed}
Suppose $Y_{k}$ is $\mathcal{F}_k$-measurable and $\mathbb{E} (Y_{k} \; | \; \mathcal{F}_{k-1} )=0$. Set $S_n=\sum_{k=1}^n Y_k$, let $M$ be a positive real number and let $T_n(z)=\sum_{k=1}^{n}\mathbb{E} (Y_{k}(x)^2 \; | \; \mathcal{F}_{k-1} )$. Suppose $\mathbb{P} ( |Y_{k}| \leq M \; \text{for all } k \leq n )=1$. Then for all positive numbers $\epsilon$ and $b$,
$$\mathbb{P} ( S_n \geq \epsilon \; \; \text{and} \; \; T_n(z)\leq b \; \text{for some } n>0 ) \leq \exp\left[\frac{-\epsilon^2}{2(M\epsilon+b)}\right].$$
\end{theorem}

We obtain the following result.
\begin{lemma}\label{3.5}
There exists a constant $\delta_0 >0$ such that, for any $T_0 > 0$ and $0 < \epsilon <  T_0$,
$$ \mathbb{P}\left(\sup_{0\leq t \leq T_0} |S_{n(t)}(x)|>\epsilon\right)\leq \exp \left(\frac{-\epsilon^2}{\delta_0 T_0\sqrt{c}}\right).$$
\end{lemma}
\begin{proof}
We know $Y_i(x)$ is a martingale difference array. Therefore, in order to apply Theorem \ref{freed} we need to find a bound on $\sum_{k=1}^{n}\mathbb{E} (Y_{k}(x)^2 \; | \; \mathcal{F}_{k-1} )$. 
By the definition of $\beta_{\nu}$,
\begin{align*}
\mathbb{E}\left(Y_i(x)^2|\mathcal{F}_{i-1}\right)=&\int_0^1\tilde{\gamma}_P(X_{i-1}(x)-\theta)^2 h_{\nu}(\theta)d\theta-\beta_{\nu}(X_{i-1}(x))^2.
\intertext{Therefore,}
\mathbb{E}\left(Y_i(x)^2|\mathcal{F}_{i-1}\right)=&\rho_0c^{\frac{3}{2}}h_{\nu}(X_{i-1}(x))-\beta_{\nu}(X_{i-1}(x))^2\\
+&\left(\int_0^1\tilde{\gamma}_P(X_{i-1}(x)-\theta)^2 h_{\nu}(\theta)d\theta-\rho_0c^{\frac{3}{2}}h_{\nu}(X_{i-1}(x))\right).
\end{align*}
From the bounds in Lemmas \ref{2.2} and \ref{2.4}, there exists a constant $\delta>0$ such that $|\beta_{\nu}(X_{i-1}(x))|\leq \delta c$ and $\left|\int_0^1\tilde{\gamma}(x-\theta)^2 h_{\nu}(\theta)d\theta-\rho_0c^{\frac{3}{2}}h_{\nu}(x)\right|\leq  \delta c^2\log (c^{-1})$.  Therefore, there exists a constant $\delta > 0$ (possibly different to that above) such that,
$$\mathbb{E}\left(Y_i(x)^2|\mathcal{F}_{i-1}\right)<\delta c^{\frac{3}{2}}.$$
Thus,
$$\sum_{i=1}^{n(t)}\mathbb{E}\left(Y_i(x)^2|\mathcal{F}_{i-1}\right)\leq \delta c^{\frac{3}{2}}n(t).$$
Using the estimates provided in Lemma \ref{2.4}, for some constant $\delta >0$, $|Y_i(x)|\leq \delta \sqrt{c}$  for all $i\geq0$  and, if $0 \leq t \leq T_0$, $\sum_{i=1}^{n(t)}\mathbb{E}\left(Y_i(x)^2|\mathcal{F}_{i-1}\right)\leq \delta T_0c^{\frac{1}{2}}.$ Putting this all into Theorem \ref{freed},
$$ \mathbb{P}\left(\sup_{0\leq t \leq T_0}|S_{n(t)}(x)|>\epsilon\right)\leq \exp \left(\frac{-\epsilon^2}{2\delta (\sqrt{c}\epsilon +T_0\sqrt{c})}\right).$$
Therefore, if $0 < \epsilon< T_0$,
$$ \mathbb{P}\left(\sup_{0\leq t \leq T_0}|S_{n(t)}(x)|>\epsilon\right)\leq \exp \left(\frac{-\epsilon^2}{4\delta T_0\sqrt{c}}\right).$$ 
\end{proof}

We now prove the main result of this section, that there exists a logarithmic time, up to which $X_{n(t)}(x)$ is uniformly close to $\psi_t(x)$ for each $x \in \mathbb{R}$. 

\begin{theorem}\label{3.6II}
Let $$ T_0 =\frac{1}{4\|b'\|_{\infty}}\left(\log(c^{-1})-3\log(\log(c^{-1}))\right).$$ Then for any $\epsilon>0$,
$$\lim_{c\to 0} \mathbb{P}\left(\sup_{0\leq t\leq T_0} \left| X_{n(t)}(x)-\psi_t(x)\right|>\epsilon\right)=0.$$
\end{theorem}
\begin{proof}
For $c$ chosen sufficiently small, 
\begin{align*}
\sup_{0\leq t\leq T_0}|X_{n(t)}(x)-\psi_t(x)|
\leq \left(\sup_{0\leq t\leq T_0}|S_{n(t)}(x)|+ \delta T_0 c^{\frac{1}{2}}\log(c^{-1})\right)+\|b'\|_{\infty} \int_{0}^{T_0} \sup_{0\leq t\leq r}|X_{n(t)}(x)-\psi_t(x)|dr.
\end{align*}
Therefore, by Gronwall's inequality \cite{Gro},
\begin{align*}
\sup_{0\leq t\leq T_0}|X_{n(t)}(x)-\psi_t(x)| \leq \left(\sup_{0\leq t\leq T_0}|S_{n(t)}(x)|+ \delta c^{\frac{1}{2}}(\log(c^{-1}))^2\right)e^{\|b'\|_{\infty}T_0}.
\end{align*}
Thus,
\begin{align*}
&\limsup_{c\to 0} \mathbb{P}\left(\sup_{0\leq t\leq T_0} \left| X_{n(t)}(x)-\psi_t(x)\right|>\epsilon\right)\\
&\leq \limsup_{c\to 0} \mathbb{P}\left(\left(\sup_{0\leq t\leq T_0}|S_{n(t)}(x)|+ \delta c^{\frac{1}{2}}(\log(c^{-1}))^2\right)>\epsilon e^{-\|b'\|_{\infty}T_0} \right)\\
&\leq\limsup_{c\to 0} \mathbb{P}\left(\sup_{0\leq t\leq T_0}|S_{n(t)}(x)|>\frac{\epsilon}{2} c^{\frac{1}{4}}(\log(c^{-1}))^{\frac{3}{4}} \right) \\
&\leq\limsup_{c\to 0}\exp\left(\frac{-\epsilon^2c^{\frac{1}{2}}(\log(c^{-1}))^{\frac{3}{2}}}{4\delta_0 T_0\sqrt{c}}\right)\\
&=\limsup_{c\to 0}\exp\left(\frac{-\epsilon^2\|b'\|_{\infty}(\log(c^{-1}))^{\frac{1}{2}}}{\delta_0}\right),
\end{align*}
where the penultimate line used Lemma \ref{3.5}.
The required result follows.
\end{proof}

\section{Analysis of fluctuations}
\label{sec:fluc}
Having shown that $X_{n(t)}(x)$ is well approximated by $\psi_t(x)$, it is natural to ask about the distribution of the fluctuations. In \cite{JTS} it is shown that 
\[
c^{-1/4}\left ( X_{n(t)}(x) - \psi_t(x) \right )
\]
converges in distribution with respect to the Skorokhod topology to the solution of a linear SDE. For the purpose of analysing the long-time behaviour of $X_{n(t)}(x)$ it turns out to be more convenient to analyse the `pulled-back' fluctuations
$$\psi_t^{-1}(X_{n(t)}(x))-x.$$

For any $t,s \geq 0$, $\psi_{t+s}(x)=\psi_t(\psi_s(x))$. Therefore $$\psi_t^{-1}(X_{n(t)}(x))=\psi_{t-n(t)c}^{-1}(\psi_{n(t)c}^{-1}(X_{n(t)}(x))).$$ Moreover, $0\leq t-n(t)c<c \to 0$ and so $\psi_{t-n(t)c}^{-1}$ converges to the identity mapping. By continuity, the pulled-back fluctuation above has the same limit as 
$$\psi_{nc}^{-1}(X_{n}(x))-x$$
with $n=n(t)$.  We will show the fluctuations are of order $c^{\frac{1}{4}}$. Therefore, for each fixed $x\in \mathbb{R}$, let 
\begin{equation}
\widetilde{Z}_n(x)=c^{-\frac{1}{4}}\left(\psi_{nc}^{-1}(X_{n}(x))-x\right).
\end{equation}
For notational simplicity we will denote $\Phi_{t}(x)=\psi_{t}^{-1}(x)$. Let $Z_t(x)$ be the solution to the stochastic differential equation,
\begin{equation}\label{sdeequation}
dZ_t(x)=\sqrt{\rho_0}\Phi_{t}'(\psi_t(x)) \sqrt{h_{\nu}(\psi_{t}(x))}dB_{t}
\end{equation}
with $Z_0(x)=0$, where $B_t$ is a standard Brownian motion.
The main result of this section is stated as follows.
\begin{theorem}\label{fluc}
The stochastic process $\widetilde{Z}_{n(t)}(x)\to Z_t(x)$ in distribution as $c\to 0$ with respect to the Skorokhod topology.
\end{theorem}
Note that as the limit process is almost surely continuous, it follows immediately that the process converges in distribution with respect to the topology of uniform convergence. The proof will consist of showing that in the limit $\widetilde{Z}_{n(t)}(x)$ and $Z_t(x)$ share the same finite dimensional distributions, together with an appropriate tightness argument. We start by evaluating the finite dimensional distributions. Observe that we can rewrite the fluctuations as the following sum,
$$\psi_{nc}^{-1}(X_{n}(x))-x=\sum_{i=1}^n \left(\Phi_{ic}(X_{i}(x))-\Phi_{(i-1)c}(X_{i-1}(x))\right).$$
The following lemma identifies the leading order terms of the fluctuations. 
\begin{lemma}\label{Zeq}
Set $$\widetilde{\mathcal{E}}_{n}(x) = \widetilde{Z}_{n}(x) -c^{-\frac{1}{4}}\sum_{i=1}^{n}\Phi_{ic}'(X_{i-1}(x))Y_i(x).$$
Then, for fixed $t>0$, $\sup_{0\leq n \leq n(t)}\left | \widetilde{\mathcal{E}}_{n}(x) \right | \to 0$ in probability as $c\to 0$.
\end{lemma}
\begin{proof}
By Taylor's theorem, there exists some reminder term $R_i(x)$ which will be bounded below satisfying
\begin{align*}
\Phi_{ic}(X_{i}(x))-\Phi_{(i-1)c}(X_{i-1}(x))
&=\Phi_{ic}'(X_{i-1}(x))(X_i(x)-X_{i-1}(x))+\dot{\Phi}_{ic}(X_{i-1}(x))c+R_i(x).
\end{align*}
Since $\Phi_t(\psi_t(x))=x$ for every $x\in \mathbb{R}$, $t\in \mathbb{R}$, taking the derivative with respect to $t$ gives 
\[
\Phi_t'(\psi_t(x))\dot{\psi}_t(x)+\dot{\Phi}_t(\psi_t(x))=0.
\]
By definition, $\dot{\psi}_t(x)=b(\psi_t(x))$. Thus,
\[
\dot{\Phi}_t(\psi_t(x))=-\Phi_t'(\psi_t(x))b(\psi_t(x)).
\]
This holds for any $x\in \mathbb{R}$, so substituting $\Phi_t(X_{i-1}(x))$ in for $x$, it follows that
\[
\dot{\Phi}_{ic}(X_{i-1}(x))=-\Phi_{ic}'(X_{i-1}(x))b(X_{i-1}(x))
\]
and so
\[
\Phi_{ic}(X_{i}(x))-\Phi_{(i-1)c}(X_{i-1}(x)) = 
\Phi_{ic}'(X_{i-1}(x))(X_i(x)-X_{i-1}(x) - c b(X_{i-1}(x))+R_i(x).
\]

Recall $X_i(x)-X_{i-1}(x)=Y_i(x)+\beta_{\nu}(X_{i-1}(x))$ so
\begin{align*}
\Phi_{ic}(X_{i}(x))-\Phi_{(i-1)c}(X_{i-1}(x))&=\Phi_{ic}'(X_{i-1}(x))Y_i(x)+ \mathcal{E}_i(x),
\end{align*}
where
\[
\mathcal{E}_i(x)=R_i(x)+\Phi_{ic}'(X_{i-1}(x))\left(\beta_{\nu}(X_{i-1}(x))-cb(X_{i-1}(x))\right).
\]
Therefore
\[ 
\widetilde{Z}_{n(t)}(x)=c^{-\frac{1}{4}}\left(\sum_{i=1}^{n(t)}\Phi_{ic}'(X_{i-1}(x))Y_i(x)+\mathcal{E}_i(x)\right).
\]
All that remains is to find upper bounds on the error $\widetilde{\mathcal{E}}_{n}(x)=c^{-\frac{1}{4}}\sum_{i=1}^{n}\mathcal{E}_i(x)$. The Taylor remainder term is given by
\[
R_i(x)=\Phi_{ic}''(\zeta)(X_i(x)-X_{i-1}(x))^2+\ddot{\Phi}_{\rho}(X_{i-1}(x))c^2
\]
for some $\zeta$ between $X_{i-1}(x)$  and $X_i(x)$ and $(i-1)c<\rho<ic$. Fix $t > 0$. Since we only require convergence in probability, we may restrict to the high probability event
\[
\left \{ \sup_{0 \leq s \leq t} |X_{n(s)} - \psi_s(x)| \leq 1 \right \}.
\]
By the definition of $\psi_t(x)$ along with the assumption that $h_{\nu}$ is twice continuously differentiable there exists a constant $\delta>0$, dependent only on $t$, such that $|\Phi_{ic}''(\zeta)|<\delta$ for all $i \leq n(t)$. Therefore, using that $$X_i(x)-X_{i-1}(x)=Y_i(x)+\beta_{\nu}(X_{i-1}(x))$$ there exists a (possibly different) constant $\delta>0$, dependent again only on $t$, such that,
$$|R_i(x)|\leq \delta \left( |Y_i(x)|^2+c^{2}\right)$$
for all $i \leq n(t)$. Similarly, using Lemma \ref{2.4}, there exists a constant $\delta>0$ (dependent only on $t$) such that 
\begin{align*}
\left|\Phi_{ic}'(X_{i-1}(x))\left(\beta_{\nu}(X_{i-1}(x))-cb(X_{i-1}(x))\right)\right|&\leq |\Phi_{ic}'(X_{i-1}(x))|\left|\beta_{\nu}(X_{i-1}(x))-cb(X_{i-1}(x))\right|\\
&\leq \delta c^{\frac{3}{2}}\log(c^{-1}).
\end{align*}
Therefore,
$$|\mathcal{E}_i(x)|\leq \delta( |Y_i(x)|^2+c^{\frac{3}{2}}\log(c^{-1}))$$
for some positive constant $\delta$ dependent only on $t$. Thus, using from the proof of Lemma \ref{3.5} that 
\[
\sum_{i=1}^{n(t)}\mathbb{E}\left(Y_i(x)^2\right)\leq \delta c^{\frac{1}{2}},
\]
we get 
\[
\mathbb{E} \left ( \sup_{0\leq n \leq n(t)}|\widetilde{\mathcal{E}}_{n}(x)| \right ) \leq \delta c^{1/4} \log(c^{-1}).
\]
The result follows by Markov's inequality.
\end{proof}

All that remains is to analyse the term $c^{-\frac{1}{4}}\sum_{i=1}^{n}\Phi_{ic}'(X_{i-1}(x))Y_i(x)$. It is immediate from the definition of $Y_i$ that this is a martingale. We will thereore apply the following result of Mcleish \cite{M}.
\begin{theorem}[McLeish]\label{mcleish}
Let $(X_{k,n})_{1\leq k \leq n}$ be a martingale difference array with respect to the filtration $\mathcal{F}_{k,n}=\sigma(X_{1,n}, X_{2,n},...,X_{k,n}). $ Let $M_{n}=\sum_{i=1}^{n}X_{i,n}$ and assume that;
\begin{itemize}
\item[$(1)$] for all $\rho>0$, $\sum_{k=1}^n X_{k,n}^2  $ $\mathbbm{1}(|X_{k,n}|>\rho)\to 0$ in probability as $n\to \infty$.
\item[$(2)$] \; $\sum_{k=1}^n X_{k,n}^2 \to s^2$ in probability as $n\to\infty$ for some $s^2>0$.
\end{itemize}
Then $M_{n}$ converges in distribution to $\mathcal{N}(0, s^2).$
\end{theorem}
In order to use this result we establish a series of lemmas.
\begin{lemma}\label{4.4}
For fixed $t>0$,
$$\sum_{i=1}^{n(t)}\mathbb{E}\left(c^{-\frac{1}{2}}\left(\Phi_{ic}'(X_{i-1}(x))Y_i(x)\right)^2|\mathcal{F}_{i-1}\right)\to \rho_0\int_0^t (\Phi_{s}'(\psi_s(x)))^2h_{\nu}(\psi_s(x))ds $$
in probability as $c\to 0$.
\end{lemma}
\begin{proof}
As $[0,t]$ is a compact time interval, by the proof of Lemma \ref{3.5} it follows that
\begin{align*}
\sum_{i=1}^{n(t)}\mathbb{E}\left(c^{-\frac{1}{2}}\left(\Phi_{ic}'(X_{i-1}(x))Y_i(x)\right)^2|\mathcal{F}_{k-1}\right)&=\rho_0\int_0^t (\Phi_{s}'(X_{n(s)}(x)))^2h_{\nu}(X_{n(s)}(x))ds + O(c^{\frac{1}{2}}\log(c^{-1})t).
\end{align*}
Then by Theorem \ref{3.6II}, 
$$\sum_{i=1}^{n(t)}\mathbb{E}\left(c^{-\frac{1}{2}}\left(\Phi_{ic}'(X_{i-1}(x))Y_i(x)\right)^2|\mathcal{F}_{k-1}\right)\to \rho_0\int_0^t (\Phi_{s}'(\psi_s(x)))^2h_{\nu}(\psi_s(x))ds$$
in probability as $c\to 0$.
\end{proof}
\begin{lemma}\label{expo}
For fixed $t>0$,
$$\sum_{i=1}^{n(t)}c^{-\frac{1}{2}}\left(\Phi_{ic}'(X_{i-1}(x))Y_i(x)\right)^2\to \rho_0\int_0^t (\Phi_{s}'(\psi_s(x)))^2h_{\nu}(\psi_s(x))ds $$
in probability as $c\to 0$.
\end{lemma}
\begin{proof}
Define  
$$\mathcal{Y}_i(z):=c^{-\frac{1}{2}}\left(\Phi_{ic}'(X_{i-1}(x))Y_i(x)\right)^2-\mathbb{E}\left(c^{-\frac{1}{2}}\left(\Phi_{ic}'(X_{i-1}(x))Y_i(x)\right)^2|\mathcal{F}_{i-1}\right) 
$$ which is a martingale difference array with respect to the filtration $(\mathcal{F}_{i})_{i\leq n}$. We need to show that $\sum_{i=1}^{n(t)}\mathcal{Y}_i(z)$ converges to zero in probability as $c \to 0$. By Markov's inequality,
\[
\mathbb{P}\left(|\sum_{i=1}^{n(t)}\mathcal{Y}_i|>\eta \right)\leq \frac{1}{\eta^2}\mathbb{E}\left(|\sum_{i=1}^{n(t)}\mathcal{Y}_i|^2\right)=\frac{1}{\eta^2}\sum_{i=1}^{n(t)}\mathbb{E}(\mathcal{Y}_i^2).
\]
Using the property that for a random variable $X$, $\mathbb{E}((X-\mathbb{E}(X))^2)\leq \mathbb{E}(X^2)$,
\[
\mathbb{P}\left(|\sum_{i=1}^{n(t)}\mathcal{Y}_i|>\eta\right)\leq \frac{1}{\eta^2}\frac{1}{c}\sum_{i=1}^{n(t)} \mathbb{E}(\left(\Phi_{ic}'(X_{i-1}(x))Y_i(x)\right)^4).
\] 
With high probability, on a compact time interval $\Phi_{ic}'(X_{i-1}(x))$ is bounded  and thus by using the bounds from the proof of Lemma \ref{3.5} and that for each $0\leq i\leq n(t)$, $|Y_i(x)|<\delta\sqrt{c}$ it follows that  $$ \mathbb{E}(\left(\Phi_{ic}'(X_{i-1}(x))Y_i(x)\right)^4)\leq \delta c^{\frac{5}{2}}$$ for some positive contant $\delta$ dependent on $t$. Thus, there exists a constant $\delta>0$ such that
$$\mathbb{P}\left(|\sum_{i=1}^{n(t)}\mathcal{Y}_i|>\eta\right)\leq \delta \frac{c^{\frac{1}{2}}t}{\eta^2} $$
which converges to zero as $c \to 0$. 
\end{proof}
By Lemma \ref{expo}, condition (2) of Theorem \ref{mcleish} is satisfied and all that remains is to show that condition (1) is also satisfied. We prove this in the form of the following lemma.
\begin{lemma}
For each $x\in \mathbb{R}$, $t>0$ fixed and for all $\rho>0$, the following statement is satisfied.
$$c^{-\frac{1}{2}}\sum_{i=1}^{n(t)} \left(\Phi_{ic}'(X_{i-1}(x))Y_i(x)\right)^2  \mathbbm{1}(|c^{-\frac{1}{4}}\Phi_{ic}'(X_{i-1}(x))Y_i(x)|>\rho)\to 0$$ in probability as $c\to 0$.
\end{lemma}
\begin{proof}
Let $\mu>0$ then, 
\begin{align*}
&\mathbb{P}\left(\sum_{i=1}^{n(t)} c^{-\frac{1}{2}}\left(\Phi_{ic}'(X_{i-1}(x))Y_i(x)\right)^2 \mathbbm{1}(|c^{-\frac{1}{4}}\Phi_{ic}'(X_{i-1}(x))Y_i(x)|>\rho)>\mu\right)\\
&\leq \mathbb{P}\left(\max_{1\leq i\leq n(t)}|c^{-\frac{1}{4}}\Phi_{ic}'(X_{i-1}(x))Y_i(x)|>\rho\right)\\
&\leq \frac{1}{\rho}\mathbb{E}\left(\max_{1\leq i\leq n(t)}|c^{-\frac{1}{4}}\Phi_{ic}'(X_{i-1}(x))Y_i(x)|\right)\\
\end{align*}
with the second inequality following by Markov's inequality. By Lemma \ref{2.4}, for all $0\leq i\leq {n(t)}$, $|Y_i(x)|<\delta \sqrt{c}$ for some positive constant $\delta$. Therefore, as on a compact time interval $\Phi_{ic}'(X_{i-1}(x))$ is bounded with high probability, there exists a constant $\delta>0$, dependent on $t$, such that,
$$\mathbb{P}\left(\sum_{i=1}^{n(t)} c^{-\frac{1}{2}}(Y_i(x))^2 \mathbbm{1}(|c^{-\frac{1}{4}}Y_i(x)|>\rho)>\mu\right)\leq \frac{1}{\rho} \delta c^{\frac{1}{4}}.$$
Thus, 
$$c^{-\frac{1}{2}}\sum_{i=1}^{n(t)} Y_i(x)^2  \mathbbm{1}(|c^{-\frac{1}{4}}Y_i(x)|>\rho)\to 0$$ in probability as $c \to 0$.
\end{proof}
Therefore, both conditions of Theorem \ref{mcleish} are satisfied. In order to show convergence in distribution of the process $(\widetilde{Z}_{n(t)}(x))_{t>0}$ all that remains is to check the covariance structure and prove that the family of processes $(\widetilde{Z}_{n(t)}(x))_{t>0}$ is tight with respect to $c$ under the Skorokhod topology \cite{B}. We know $(Z_t(x))_{t>0}$ has independent increments so we start by analysing the covariance structure of $(\widetilde{Z}_{n(t)}(x))_{t>0}$ in the limit. 
\begin{lemma}
Suppose $0\leq t_1<t_2$. Then
$$\mathrm{Cov}\left(\widetilde{Z}_{n(t_2)}(x)-\widetilde{Z}_{n(t_1)}(x),\widetilde{Z}_{n(t_1)}(x)\right)\to 0$$
as $c\to 0$.
\end{lemma}
\begin{proof}
Set
\[
\widetilde{W}_{n}(x) = c^{-\frac{1}{4}}\sum_{i=1}^{n}\Phi_{ic}'(X_{i-1}(x))Y_i(x).
\]
By Lemma \ref{Zeq}, 
\[
\lim_{c \to 0} \mathrm{Cov}\left(\widetilde{Z}_{n(t_2)}(x)-\widetilde{Z}_{n(t_1)}(x),\widetilde{Z}_{n(t_1)}(x)\right)
= \lim_{c \to 0} \mathrm{Cov}\left(\widetilde{W}_{n(t_2)}(x)-\widetilde{W}_{n(t_1)}(x),\widetilde{W}_{n(t_1)}(x)\right). 
\]
But the right-hand side is zero, using the tower law and that $\mathbb{E}(Y_k(x)|\mathcal{F}_{k-1})=0$.
\end{proof}
Therefore, in the limit, the process $(\widetilde{Z}_{n(t)}(x))_{t>0}$ shares the same covariance structure as $(Z_{n(t)})_{t>0}$ and hence we have convergence of finite dimensional distributions. All that remains before we can prove convergence as a process is to establish tightness. 
\begin{lemma}\label{4.5}
The family of processes $(\widetilde{Z}_{n(t)}(x))_{t>0}$ is tight with respect to $c$.
\end{lemma}
\begin{proof}
It is sufficient to show that Aldous's condition holds (see for example \cite[Theorem 16.10]{B}). Explicitly, we need to show that, for each $x$, and for $T>0$ not dependent on $c$,
$$\lim_{R\to \infty} \left(\sup_{0\leq t<T} \mathbb{P}\left(|\widetilde{Z}_{n(t)}(x)|\geq R\right)\right)=0$$
and if $\tau_t$ is a stopping time and $\delta_t$ converges to 0 as $c\to 0$ then,
$$|\widetilde{Z}_{n(\tau_t+\delta_t)}(x)-\widetilde{Z}_{n(\tau_t)}(x)|\to 0$$
in probability as $c\to0$. By Lemma \ref{Zeq}, it suffices for the first condition to show that  
$$\lim_{R\to \infty} \sup_{0\leq t<T} \mathbb{P}\left(\left|c^{-\frac{1}{4}}\sum_{i=1}^{n(t)} \Phi_{ic}'(X_{i-1}(x))Y_i(x)\right|\geq R\right)=0.$$
Since $Y_i(x)$ is a martingale difference array, by Markov's inequality,
\begin{align*}
\sup_{0\leq t<T} \mathbb{P}\left(\left|c^{-\frac{1}{4}}\sum_{i=1}^{n(t)} \Phi_{ic}'(X_{i-1}(x))Y_i(x)\right|\geq R\right)
&\leq \sup_{0\leq t<T}\frac{1}{R^2}\mathbb{E}\left(c^{-\frac{1}{2}}\left(\sum_{i=1}^{n(t)}\Phi_{ic}'(X_{i-1}(x)) Y_i(x) \right)^2\right)\\
&= \frac{1}{R^2}\left(\sum_{i=1}^{n(T)} \mathbb{E}(c^{-\frac{1}{2}}(\Phi_{ic}'(X_{i-1}(x))Y_i(x))^2 )\right).
\end{align*}
By Lemma \ref{4.4},
\[
\sum_{i=1}^{n(T)}\mathbb{E}\left(c^{-\frac{1}{2}}\left(\Phi_{ic}'(X_{i-1}(x))Y_i(x)\right)^2\right)\to \rho_0\int_0^T (\Phi_{s}'(\psi_s(x)))^2h_{\nu}(\psi_s(x))ds < \infty.
\]
Consequently, if we take the limit as $R\to \infty$ then the upper bound converges to $0$ and we have proved the first condition. For the second condition, it is sufficient to show that for any $\epsilon>0$ and for all $0<t_1<t_2$ where $(t_2-t_1)\to 0$ as $c\to 0$,
$$\lim_{c\to 0}\mathbb{P}\left( \sup_{t_1<t<t_2}|\widetilde{Z}_{n(t)}(x)-\widetilde{Z}_{n(t_1)}(x)|>\epsilon\right)=0.$$
We use a similar approach to the first condition. Using the bounds provided above and Markov's inequality,
\[
\lim_{c\to 0} \mathbb{P}\left( \sup_{t_1<t<t_2}|\widetilde{Z}_{n(t)}(x)-\widetilde{Z}_{n(t_1)}(x)|>\epsilon\right)
\leq \lim_{c\to 0} \frac{1}{\epsilon^2}\mathbb{E}\left(c^{-\frac{1}{2}}\left(\sum_{i=n(t_1)}^{n(t_2)} \Phi_{ic}'(X_{i-1}(x))Y_i(x) \right)^2\right).
\]
By taking conditional expectations and using the same arguments as above,
\[
\lim_{c\to 0} \mathbb{P}\left( \sup_{t_1<t<t_2}|\widetilde{Z}_{n(t)}(x)-\widetilde{Z}_{n(t_1)}(x)|>\epsilon\right)
\leq  \lim_{c\to 0} \delta (t_2-t_1) \frac{1}{\epsilon^2} = 0.
\]
Here $\delta>0$ is some constant and, for the last equality, we used that $(t_2-t_1)\to 0$ as $c\to 0$. The result follows.
\end{proof}

\section{Analysis of critical time window}
\label{sec:crit}
In Section \ref{sec:logconv} we showed that the harmonic measure flow $X_{n(t)}(x)$ converges to the the solution of the ODE given in \eqref{ode1}, $\psi_t(x)$, provided that $$0\leq t\leq \frac{1}{4\|b'\|_{\infty}}\left(\log(c^{-1})-3\log(\log(c^{-1}))\right).$$ In this section we suppose that the ODE has an unstable fixed point $a_u$. We show there exists a critical time window during which $X_{n(t)}(a_u)$ moves a macroscopic distance away from $a_u=\psi_t(a_u)$. 

Fix $t_0 > 0$. Our strategy will be to compare $X_t(a_u)$ to $\psi_{t-t_0}(X_{k_0}(a_u))$ when $t_0$ is large, where $k_0=n(t_0)$. Observe that if $t \geq t_0$, then 
 \[
\psi_{t-t_0}(X_{k_0}(a_u))=\psi_{t}(\psi_{t_0}^{-1}(X_{k_0}(a_u)))
=\psi_{t}(a_u+c^{\frac{1}{4}}\widetilde{Z}_{n(t_0)}(a_u)).
\] 
We sometimes use the expression $\psi_{t-t_0}(X_{k_0}(a_u))$ when $t < t_0$. In this case, the right hand side of the equation above can be used to interpret what we mean by this. 
In the previous section we showed $\widetilde{Z}_{n(t_0)}(a_u)\to Z_{t_0}(a_u)$ in distribution as $c\to 0$ where $Z_{t_0}(a_u)$ is Gaussian with mean zero and variance  given by $$\rho_0\int_0^{t_0} (\Phi_{s}'(\psi_s(a_u)))^2h_{\nu}(\psi_s(a_u))ds = \rho_0\int_0^{t_0} (\Phi_{s}'(a_u))^2h_{\nu}(a_u)ds.$$ By definition we know $\Phi_t(\psi_t(x))=x$. Therefore, by the chain rule $\Phi_t'(\psi_t(x))=(\psi_t'(x))^{-1}$. Furthermore, by the definition of $\psi_t$,  
\[
(\dot{\psi}_t)'(x)=b'(\psi_t(x))\psi_t'(x)
\]
with $\psi'(x) = 1$ and hence
\begin{equation}
\label{derivpsi}
\psi'(x) = \exp \left ( \int_0^t b'(\psi_s(x)) ds \right ).
\end{equation}
Set $b'(a_u) = \lambda_u > 0$. As $a_u$ is a stationary point,
$$ \psi_t'(a_u)=e^{\lambda_u t}.$$
It follows that the variance of $Z_{t_0}(a_u)$ is given by,
$$\rho_0\int_0^{t_0} e^{-2\lambda_u s} h_{\nu}(a_u)ds=\frac{\rho_0 h_{\nu}(a_u)}{2\lambda_u}\left(1-e^{-2\lambda_u t_0}\right).$$
For the remainder of this section we assume $h_{\nu}(a_u)>0$ so that the variance above is finite and non-zero. This assumption is not very restrictive because if $h_{\nu}(a_u)=0$ we can replace the density $h_{\nu}(x)$ with $\frac{h_{\nu}(x)+\delta}{1+\delta}$ for some small constant $\delta>0$, which would in turn replace $b(x)$ with $\frac{b(x)}{1+\delta}$. In particular, the fixed points remain in the same location and share the same stability properties. 

For notational convenience we will assume $\widetilde{Z}_{n(t)}$ and $Z_{t}$ are constructed on the same probability space. Hence, by restricting to a subsequence of $c$'s if necessary, we may assume that the stochastic process $\widetilde{Z}_{n(t)}(a_u)\to Z_{t}(a_u)$ in probability as $c\to 0$ with respect to the Skorokhod topology. Furthermore, by the $L^2$ martingale convergence theorem, $Z_{t_0}(a_u)\to Z_{\infty}(a_u)$ in $L^2$, and hence in probability, as $t_0\to \infty$. It follows that for all $\epsilon>0$,
\[
\lim_{t_0 \to 0} \lim_{c \to 0} \mathbb{P} \left ( |\widetilde{Z}_{n(t)}(a_u) -  Z_{\infty}(a_u)| > \epsilon \right ) = 0.
\]

The computations in the previous sections suggest a change in the behaviour of $X_{n(t)}(a_u)$ on a window around $t\approx \frac{1}{4\lambda_u}\log(c^{-1})$. Thus, the remainder of this section focuses on analysing the behaviour of $X_{n(t)}(a_u)$ on this time-scale. In Section \ref{subsec:anal} we construct a random time $T_1^*$ with the property that for any fixed $0 < T < T_1^*$, $$\sup_{t \in [T_1^*-T, T_1^*+T]} \left|X_{n(t)}(a_u)-\psi_{t}\left(a_u+c^{\frac{1}{4}}\widetilde{Z}_{n(t_0)}(a_u)\right)\right|\to 0$$ in probability as $c\to 0$ and then $t_0\to \infty$. Then in Section \ref{sec4.2} we analyse the stopping time $T_1^*$ and show that $T_1^*\approx \frac{1}{4\lambda_u}\log(c^{-1})$ and by this time $X_{n(t)}(a_u)$ has moved a macroscopic distance away from the unstable trajectory and towards a stable trajectory. 

\subsection{Convergence of the stopped process}
\label{subsec:anal}

The strategy for analysing the long time behaviour of $X_t(x) - \psi_t(x)$ is to study the process $\Phi'_t(x)(X_t(x) - \psi_t(x))$. How this process behaves is strongly dependent on the sign of $b'(\psi_t(x))$. In this section we analyse this process when $\psi_t(x)$ is close to the unstable point $a_u$, so $b'(\psi_t(x)) > 0$; in Section \ref{sec:convtraj} we will consider the case when $x$ is close to a stable point, so $b'(\psi_t(x)) < 0$.

Since $b(x)$ is twice continuously differentiable, there exists an interval $(x_-, x_+)$ containing $a_u$ such that $b'$ is monotone on $[x_-, a_u]$ and on $[a_u, x_+]$ (strict monotonicity is not required). We may assume further that $x_\pm$ are sufficiently close to $a_u$ such that $\lambda_u/2 \leq b'(x) \leq 3 \lambda_u/2$ for all $x \in [x_-, x_+]$ and 
\[
\max \{ a_u - x_-, x_+-a_u\} \leq \frac{\lambda_u}{8 \|b''\|_\infty}.
\]
For each $x \in (x_-, x_+)$, let $T(x) = \inf\{t \geq 0: \psi_t(x) \notin [x_-, x^+] \}$. By \eqref{derivpsi}, provided $t \leq T(x)$, $\psi'_t(x)$ is increasing in $t$ and
\[
e^{\lambda_u t/2} \leq \psi_t'(x) \leq e^{3\lambda_u t/2}.
\] 
More generally,
\begin{equation}\label{psiinfbound}
e^{-\|b'\|_\infty t} \leq \|\psi_t' \|_{\infty}\leq e^{\|b'\|_{\infty} t}.
\end{equation}

As above, fix some time $t_0>0$ and let $k_0=n(t_0)$. 
Let
$$I(t_1,t_2)=\int_{t_1}^{t_2} b'(\psi_{s-t_0}(X_{k_0}(a_u))ds$$
so
\[
\psi_t'(a_u+c^{\frac{1}{4}}\widetilde{Z}_{n(t_0)}(a_u)) = e^{I(0,t)}.
\]
Set
$$g(t,y)= e^{-I(0,t)}\left(y-\psi_{t-ck_0}(X_{k_0}(a_u))\right)$$
and define the stopping time 
$$\mathcal{T}=\inf\{s \geq t_0 :\; |g(s,X_{n(s)}(a_u))|>c^{\frac{1}{4}}e^{-\frac{3}{4}I(0,t_0)}\}. $$
In order to prove that $\left|X_{n(t)}-\psi_{t-ck_0}(X_{k_0}(a_u))\right|$ is small during the critical time interval, we will need to control when $\psi_{t-t_0}(X_{k_0}(a_u))$ leaves the interval $[x_-,x_+]$, so set
$$T_1 = \inf\left\lbrace t\geq t_0: \psi_{t-t_0}(X_{k_0}(a_u))\not \in [x_-,x_+]\right\rbrace.$$
Note that, for $t_0 \leq t_1 < t_2 \leq T_1$, 
\begin{equation}
\label{eq:Ibound}
0 < \lambda_u (t_2-t_1)/2 \leq I(t_1,t_2) \leq 3\lambda_u (t_2-t_1)/2.
\end{equation}
Our proof will require that both $T_1$ and $e^{I(0,t)}$ are not too large and thus we introduce a further random time
\begin{equation}\label{Tstarbound}
 T_1^*=\min\left(T_1,\; \inf\left\lbrace t \geq t_0 : e^{I(0,t)}>c^{-\frac{1}{4}}e^{\frac{I(0,t_0)}{8}}\right\rbrace,  c^{-\frac{1}{2}} \right).
\end{equation}
In this section, we will show that with high probability $\mathcal{T} > T_1^*$.
In Section \ref{sec4.2} we will analyse $T_1^*$ and show that with high probability the process $\psi_{t-t_0}(X_{k_0}(a_u))$ leaves the interval $[x_-,x_+]$ before  either of the other two upper bounds in the definition of $T_1^*$ and hence $T_1^*=T_1$. 

Write
\begin{align}\label{gequation}
g(nc,X_{n}(a_u))=& M(a_u,n)+L(a_u,n)
+\sum_{i=k_0}^{n-1}\left(cb'(\psi_{ic-ck_0}(X_{k_0}(a_u))-I((i-1)c,ic)\right)g(ic,X_{i}(a_u)))
\end{align}
where 
$$M(a_u,n):=\sum_{i=k_0}^{n-1}e^{-I(0, ic)} Y_{i+1}(a_u)$$
is a martingale term that we will show is small in Lemma \ref{Mbound} and $L(a_u,n)$ is a remainder term which we will show is small in Lemma \ref{Lbound}. 

\begin{lemma}\label{Mbound}
We have
$$\lim_{t_0\to \infty}\lim_{c\to 0}\mathbb{P}\left( \sup_{t_0\leq t\leq T_1^*}|M(a_u,n(t))|>c^{\frac{1}{4}}e^{-\frac{7}{8}I(0,t_0)}\right)= 0.$$
\end{lemma}
\begin{proof}
The proof uses Theorem \ref{freed}. First, 
by Lemma \ref{2.4}, there exists a constant $\delta>0$ such that, 
for every $t_0\leq t\leq T_1^*$ and $k_0+1\leq i\leq n(t)$ 
$$|e^{-I(0, (i-1)c)}Y_i(a_u)|\leq \delta e^{-I(0, t_0)}\sqrt{c}$$ for sufficiently small $c$. Furthermore, in the previous sections we have shown for each $x$,
\begin{align*}
\mathbb{E}\left(Y_i(x)^2|\mathcal{F}_{i-1}\right)=&\rho_0c^{\frac{3}{2}}h_{\nu}(X_{i-1}(x))-\beta_{\nu}(X_{i-1}(x))^2\\
& + \left(\int_0^1\widetilde{\gamma}_P(X_{i-1}(x)-\theta)^2 h_{\nu}(\theta)d\theta-\rho_0c^{\frac{3}{2}}h_{\nu}(X_{i-1}(x))\right).
\end{align*}
From the bounds in Lemmas \ref{2.2} and \ref{2.4} there exists a constant $\delta>0$ such that $|\beta_{\nu}(X_{i-1}(x))|\leq \delta c$ and $\left|\int_0^1\tilde{\gamma}(x-\theta)^2 h_{\nu}(\theta)d\theta-\rho_0c^{\frac{3}{2}}h_{\nu}(x)\right|\leq  \delta c^2\log (c^{-1})$. Therefore for $t_0\leq t\leq T_1^*$ and $0\leq i\leq n(t)$,
\begin{align*}
\mathbb{E}\left(e^{-2I(0, (i-1)c)}Y_i(a_u)^2|\mathcal{F}_{i-1}\right)
&\leq e^{-2I(0, (i-1)c)}\rho_0c^{\frac{3}{2}}h_{\nu}(X_{i-1}(a_u))+e^{-2I(0, (i-1)c)}\delta(c^2+c^2\log (c^{-1}))\\
&\leq 2\rho_0h_{\nu}(X_{i-1}(a_u))c^{\frac{3}{2}}e^{-2I(0, (i-1)c)}
\end{align*}
for sufficiently small $c$.
Therefore,
\[
\sum_{i=k_0+1}^{n(t)}\mathbb{E}\left(e^{-2I(0, (i-1)c)}Y_i(a_u)^2|\mathcal{F}_{i-1}\right)\leq 2\rho_0\|h_{\nu}\|_{\infty}c^{\frac{1}{2}}\sum_{i=k_0+1}^{n(t)}ce^{-2I(0, (i-1)c)}
\]
we can approximate this sum with a Riemann integral to show
\[
\sum_{i=k_0+1}^{n(t)}\mathbb{E}\left(e^{-2I(0, (i-1)c)}Y_i(a_u)^2|\mathcal{F}_{i-1}\right)\leq \frac{2\rho_0\|h_{\nu}\|_{\infty}}{\lambda_u}c^{\frac{1}{2}}e^{-2I(0, t_0)}
\]
for $c$ sufficiently small. 
The result follows by Theorem \ref{freed}.
\end{proof}
\begin{lemma}\label{Lbound}
Let $L(a_u,n(t))$ be defined by equation \eqref{gequation}.
Then
$$ \sup_{t_0\leq t\leq T_1^* \wedge \mathcal{T}} |L(a_u,n(t))|<c^{\frac{1}{4}}e^{-\frac{7}{8}I(0,t_0)}.$$
\end{lemma}
\begin{proof}
Write $g(nc, X_{n}(a_u))$ as a telescopic sum,
\begin{align*}
g(nc, X_{n}(a_u))=&\sum_{i=k_0}^{n-1} \left(g((i+1)c, X_{i+1}(a_u))-g(ic, X_{i}(a_u))\right)\\
=&\sum_{i=k_0}^{n-1} (1-e^{I(ic,(i+1)c)})g((i+1)c, X_{i+1}(a_u))\\
&+\sum_{i=k_0}^{n-1} \left(e^{I(ic,(i+1)c)}g((i+1)c, X_{i+1}(a_u))-g(ic, X_{i}(a_u))\right).
\end{align*}
By Taylor expanding $(1-e^{I(ic,(i+1)c)})$ the first summation can be written as 
\begin{align*}
-\sum_{i=k_0}^{n-1} I((i-1)c,ic)g(ic, X_{i}(a_u))+R_1(a_u,n)
\end{align*}
where
\begin{align*}
R_1(a_u,n) =-&\sum_{i=k_0}^{n-1} \frac{e^{\zeta_i}}{2}I(ic,(i+1)c)^2g((i+1)c, X_{i+1}(a_u))\\
&-I((n-1)c,nc)g(nc, X_{n}(a_u))+I((k_0-1)c,k_0c)g(k_0c, X_{k_0}(a_u))
\end{align*}
for some $0\leq \zeta_i\leq I(ic,(i+1)c)$. Therefore, for $t \leq T_1^* \wedge\mathcal{T}$,
\begin{align*}
\left|R_1(a_u,n(t))\right|&\leq \frac{3\lambda_u e^{3\lambda_u c/2}}{4}c^{\frac{5}{4}}e^{-\frac{3}{4}I(0,t_0)}\sum_{i=k_0}^{n(t)-1} I(ic,(i+1)c) + 3 \lambda_u c^{\frac{9}{4}}e^{-\frac{3}{4}I(0,t_0)} \\
&\leq \frac{3\lambda_u e^{3\lambda_u c/2}}{4}c^{\frac{5}{4}}e^{-\frac{3}{4}I(0,t_0)} I(t_0,t) + 3 \lambda_u c^{\frac{9}{4}}e^{-\frac{3}{4}I(0,t_0)}\\
&\leq  \frac{3\lambda_u e^{3\lambda_u c/2}}{2}c^{\frac{5}{4}} e^{-\frac{3}{4}I(0,t_0)}\left(\log(c^{-\frac{1}{4}})+\frac{ 3\lambda_u t_0}{16}\right)
\end{align*}
where the last inequality follows from the definition of $T_1^*$. 

For the second term,
\begin{align*}
&e^{I(ic,(i+1)c)}g((i+1)c, X_{i+1}(a_u))-g(ic, X_{i}(a_u)) \\
= &e^{-I(0,ic)}\left ( (X_{i+1}(a_u) - X_i(a_u))-  (\psi_{(i+1)c-ck_0}(X_{k_0}(a_u) - \psi_{ic-ck_0}(X_{k_0}(a_u)) \right ) \\
= & e^{-I(0,ic)}\left ( Y_{i+1}(a_u) + \beta_\nu(X_i(a_u)) -  cb(\psi_{ic-ck_0}(X_{k_0}(a_u))) +  c^2\ddot{\psi}_{\rho_i}(X_{k_0}(a_u)) \right )
\end{align*}
for some $ic - ck_0 \leq \rho_i\leq (i+1)c - ck_0$.
Therefore
\begin{align*}
&\sum_{i=k_0}^{n-1} \left(e^{I(ic,(i+1)c)}g((i+1)c, X_{i+1}(a_u))-g(ic, X_{i}(a_u))\right)\\
=&M(a_u,n)+\sum_{i=k_0}^{n-1}e^{-I(0, ic)}\left(\beta_{\nu}(X_{i}(a_u))
-cb(\psi_{ic-ck_0}(X_{k_0}(a_u)))\right)+R_2(a_u,n)
\end{align*}
where 
\[
|R_2(a_u,n(t))|\leq \|b\|_{\infty}\|b'\|_{\infty}c\sum_{i=k_0}^{n(t)-1}ce^{-I(0, ic)}.
\]
By approximating the summation by a Riemann integral we see that,
\[
\sup_{t_0\leq t\leq T_1^* \wedge\mathcal{T}}|R_2(a_u,n(s))|\leq \frac{2\|b\|_{\infty}\|b'\|_{\infty}}{\lambda_u}ce^{-I(0,t_0)}.
\]
Set
$$R_3(a_u,n)=\sum_{i=k_0}^{n-1}e^{-I(0, ic)}\left(\beta_{\nu}(X_{i}(a_u))-cb(X_{i}(a_u))\right)$$
and
$$R_4(a_u,n)=c \sum_{i=k_0}^{n-1}e^{-I(0, ic)}\left(b(X_{i}(a_u))- b(\psi_{ic-ck_0}(X_{k_0}(a_u))) \right).$$
Using the same Riemann integral approximation as above along with the bound from Lemma \ref{2.4}, 
we see that
$$\sup_{t_0\leq t\leq T_1^* \wedge\mathcal{T}}|R_3(a_u,n(t))|\leq \frac{2\delta c^{\frac{1}{2}}\log(c^{-1})}{\lambda_u}e^{-I(0,t_0)}.$$
Taylor expanding the expression inside the summation  
$$R_4(a_u,n)=\sum_{i=k_0}^{n-1}e^{-I(0, ic)}cb''(\mu_i)(X_{i}(a_u)-\psi_{ic-ck_0}(X_{k_0}(a_u)))^2,$$ for some $\mu_i$ between $\psi_{ic-ck_0}(X_{k_0}(a_u))$ and $X_{i}(a_u)$.
For $t \leq T_1^* \wedge\mathcal{T}$ 
$$|X_{i}(a_u)-\psi_{ic-ck_0}(X_{k_0}(a_u))| = |g(ic,X_{i}(a_u))|e^{I(0,ic)}<c^{\frac{1}{2}}e^{-\frac{3}{2}I(0,t_0)}e^{I(0,ic)}.$$
Therefore, 
\[
\sup_{t_0\leq s\leq t}|R_4(a_u,n(s))|\leq\|b''\|_{\infty}c^{\frac{1}{2}}e^{-\frac{3}{2}I(0,t_0)}\sum_{i=k_0}^{n(t)-1}ce^{I(0,ic)}.
\]
We can approximate this sum with a Riemann integral again to reach the upper bound,
\[
\sup_{t_0\leq s\leq t}|R_4(a_u,n(s))|\leq\frac{2\|b''\|_{\infty}}{\lambda_u}c^{\frac{1}{2}}e^{-\frac{3}{2}I(0,t_0)}e^{I(0,t)}.
\]
However, since $t_0\leq t\leq T_1^*$, we have $e^{I(0,t)}\leq c^{-\frac{1}{4}}e^{\frac{I(0,t_0)}{8}}$. Thus,
\[
\sup_{t_0\leq s\leq t}|R_4(a_u,n(s))|\leq\frac{2\|b''\|_{\infty}}{\lambda_u}c^{\frac{1}{4}}e^{-\frac{11}{8}I(0,t_0)}.
\]
Combining all the summations above we see that 
$L(a_u,n)=R_1(a_u,n) + R_2(a_u,n)+R_3(a_u,n)+R_4(a_u,n)$. The result follows.
\end{proof}
\begin{theorem}\label{4.2}
For $\mathcal{T}$ and $T_1^*$ defined as above,
$$
\lim_{t_0\to \infty}\lim_{c\to 0}\mathbb{P}\left( \mathcal{T} > T_1^* \right ) = 1
$$
and
$$\lim_{t_0\to \infty}\lim_{c\to 0}\mathbb{P}\left(\sup_{t_0\leq t\leq T_1^*}\left|X_{n(t)}\left(a_u\right)-\psi_{t-ck_0}\left(X_{k_0}\left(a_u\right)\right)\right|>e^{-\frac{1}{2}I(0,t_0)}\right)=0.$$
\end{theorem}
\begin{proof}
We begin by obtaining an estimate on the last term in \eqref{gequation}. Observe that
\begin{align*}
\left|cb'(\psi_{ic-ck_0}(X_{k_0}(a_u))-I((i-1)c,ic)\right|&=\left|cb'(\psi_{ic-ck_0}(X_{k_0}(a_u))-\int_{(i-1)c}^{ic} b'(\psi_{r-t_0}(X_{k_0}(a_u))dr\right|\\
&\leq \int_{(i-1)c}^{ic}\left|b'(\psi_{ic-ck_0}(X_{k_0}(a_u))-b'(\psi_{r-t_0}(X_{k_0}(a_u))\right|dr\\
&\leq \|b''\|_{\infty}\|b\|_{\infty}\int_{(i-1)c}^{ic}\left|ic-r+ t_0 - ck_0 \right|dr\\
&\leq 2 c^2 \|b''\|_{\infty}\|b\|_{\infty}
\end{align*}
where the penultimate inequality follows from the Mean Value Theorem. Thus, 
\begin{align*}
|g(t,X_{n(t)}(x))|\leq &\sup_{t_0\leq s\leq t}|M(a_u,n(s))|+\sup_{t_0\leq s\leq t}|L(a_u,n(s))|\\
&+c \|b''\|_{\infty}\|b\|_{\infty}\int_{t_0}^{t}|g(s,X_{n(s)}(a_u)))|ds.
\end{align*}
By Lemma's \ref{Mbound} and \ref{Lbound},
$$\lim_{t_0\to \infty}\lim_{c\to 0}\mathbb{P}\left( \left(\sup_{t_0\leq t< \mathcal{T} \wedge T_1^*}|M(a_u,n(t))|+\sup_{t_0\leq t< \mathcal{T} \wedge T_1^*}|L(a_u,n(t))|\right)> 2 c^{\frac{1}{4}}e^{-\frac{7}{8}I(0,t_0)}\right)= 0.$$
Therefore by Gronwall's inequality, if $t_0\leq t< \mathcal{T} \wedge T_1$ then with high probability,
$$|g(t,X_{n(t)}(x))|\leq  2 c^{\frac{1}{4}}e^{-\frac{7}{8}I(0,t_0)}e^{(t-t_0)c \|b''\|_{\infty}\|b\|_{\infty}}.$$
However, using that $T_1 \leq c^{-1/2}$, we have $2 c^{\frac{1}{4}}e^{-\frac{7}{8}I(0,t_0)}e^{(t-t_0)c \|b''\|_{\infty}\|b\|_{\infty}}<c^{\frac{1}{4}}e^{-\frac{3}{4}I(0,t_0)}$ if $t_0$ is sufficiently large and $c$ sufficiently small. Thus with high probability the stopping time $\mathcal{T} \neq \mathcal{T} \wedge T_1^*$ and  
\begin{align*}
\lim_{t_0\to \infty}\lim_{c\to 0}\mathbb{P}\left(\sup_{t_0\leq t\leq T_1^*}\left|X_{n(t)}\left(a_u\right)-\psi_{t-t_0}\left(X_{k_0}\left(a_u\right)\right)\right|>e^{-\frac{1}{2}I(0,t_0)}\right)=0.
\end{align*}

\end{proof}

\subsection{Convergence of the stopping time}\label{sec4.2}
The aim of this section is to analyse $T_1^*$, defined in \eqref{Tstarbound}, and show that with high probability it is close to $\frac{1}{4\lambda_u}\log(c^{-1})$ plus an error term which is tight in $c$. 
\begin{theorem}\label{4.4a}
For $T_1$ and $T_1^*$ defined as above
$$\lim_{t_0\to \infty}\lim_{c\to 0}\mathbb{P}\left( T_1^* = T_1 \right)=1.$$
Moreover, for any $\epsilon>0$ there exists a constant $\widetilde{T_\epsilon}>0$ such that
$$\lim_{t_0\to \infty}\lim_{c\to 0}\mathbb{P}\left(\left|T_1^*-\frac{1}{4\lambda_u}\log (c^{-1})\right|>\widetilde{T_\epsilon}\right)<\epsilon.$$
\end{theorem}
\begin{proof}
Since $Z_{\infty}(a_u)$ is a Gaussian with mean $0$ and finite non-zero variance it follows that for every $\epsilon>0$ there exists a constant $C_{\epsilon}>0$ such that 
\begin{equation}
\mathbb{P}\left( C_{\epsilon}<|Z_{\infty}(a_u)|< \frac{1}{C_{\epsilon}}\right)>1-\frac{\epsilon}{3}.\label{equation14}
\end{equation}
But we know that  $Z_{t_0}(a_u)\to Z_{\infty}(a_u)$ in distribution as $t_0\to \infty$, thus if $t_0$ is sufficiently large then, 
\[
\mathbb{P}\left( C_{\epsilon}<|Z_{t_0}(a_u)|< \frac{1}{C_{\epsilon}}\right)>1-\frac{2\epsilon}{3}.
\]
Moreover, by Theorem \ref{fluc}, $\widetilde{Z}_{n(t_0)}(a_u)\to Z_{t_0}(a_u)$ in distribution as $c\to 0$, thus if $c$ is sufficiently small then,  
\begin{equation}
\mathbb{P}\left( C_{\epsilon}<|\widetilde{Z}_{n(t_0)}(a_u)|< \frac{1}{C_{\epsilon}}\right)>1-\epsilon.\label{equation16}
\end{equation}
Note that $C_\epsilon$ is only dependent on the choice of $\epsilon$ and not on $c$ or $t_0$.
For the remainder of the proof, we will assume that we are on the event $\{ C_{\epsilon} < \widetilde{Z}_{n(t_0)}(a_u) < C_{\epsilon}^{-1}\}$; the negative case is similar. We will show that on this event, provided $c$ is sufficiently small and $t_0$ is sufficiently large, $T_1^* = T_1$ and $T_1 - \frac{1}{4 \lambda_u} \log(c^{-1})$ lies in some interval which does not depend on $c$ or $t_0$. 

By construction of $x_+$, $b'(x)$ is monotone on $[a_u, x_+]$. We will first establish the results in the case when $b'(x)$ is decreasing.
By Taylor expansion,
\begin{align*}
\psi_t(a_u+c^{\frac{1}{4}}\widetilde{Z}_{n(t_0)}(a_u))&=\psi_t(a_u)+\psi_t'(\eta)c^{\frac{1}{4}}\widetilde{Z}_{n(t_0)}(a_u)\\
&=a_u+\psi_t'(\eta)c^{\frac{1}{4}}\widetilde{Z}_{n(t_0)}(a_u)
\end{align*}
for some $a_u<\eta<a_u+c^{\frac{1}{4}}\widetilde{Z}_{n(t_0)}(a_u)$. 
By \eqref{derivpsi}, and using that $\psi_s(x)$ is increasing in $x$, if $t \leq T_1$, then 
\[
e^{I(0,t)} = \psi'_t(a_u+c^{\frac{1}{4}}\widetilde{Z}_{n(t_0)}(a_u)) \leq \psi'_t(\eta) \leq \psi'_t(a_u) = e^{\lambda_u t}
\]
and hence
\begin{equation}
\label{eq:tayexp}
a_u+e^{I(0,t)} c^{\frac{1}{4}}\widetilde{Z}_{n(t_0)}(a_u) \leq \psi_t(a_u+c^{\frac{1}{4}}\widetilde{Z}_{n(t_0)}(a_u))\leq a_u+e^{\lambda_u t} c^{\frac{1}{4}}\widetilde{Z}_{n(t_0)}(a_u).
\end{equation}
It immediately follows from the right-hand side that
\begin{equation}
\label{eq:T1lowbd}
T_1 \geq \frac{1}{\lambda_u}\log\left(\frac{x_+-a_u}{c^{\frac{1}{4}}\widetilde{Z}_{n(t_0)}(a_u)}\right) \geq \frac{1}{4 \lambda_u} \log(c^{-1}) + \log((x_+-a_u)C_\epsilon).
\end{equation}
The left-hand side, together with the bound that $e^{I(0,t)} \geq e^{\lambda_ut/2}$ for $t \leq T_1$, gives the crude estimate
\[
T_1 \leq \frac{1}{2\lambda_u}\log(c^{-1})+ \log((x_+-a_u)/C_\epsilon) < c^{-1/2}
\]
for $c$ sufficiently small and hence we can discard the $c^{-\frac{1}{2}}$ upper bound in the definition of $T_1^*$ in \eqref{Tstarbound}.
Now suppose that there exists an $t_0\leq t<T_1$ such that $e^{I(0,t)}>c^{-\frac{1}{4}}e^{\frac{I(0,t_0)}{8}}$ then, using the left-hand side of \eqref{eq:tayexp} again,
\[
a_u+e^{\frac{I(0,t_0)}{8}} C_{\epsilon} < x_+
\]
However, as $c\to 0$, $e^{\frac{I(0,t_0)}{8}}\to e^{\frac{\lambda_u t_0}{8}}$, leading to a contradiction when $t_0 > \frac{8}{\lambda_u}\log((x_+-a_u)/C_{\epsilon})$ and so, provided $c$ is sufficiently small and $t_0$ is sufficiently large, $T_1=T_1^*$.  

Now we show that, as $c\to 0$ and $t_0\to \infty$, $T_1^*$ is within a compact time of $\frac{1}{4\lambda_u}\log(c^{-1})$. The lower bound follows by equation \eqref{eq:T1lowbd} so we just need to find the upper bound. 

Taylor expanding around $a_u$ gives
\[
a_u+e^{\lambda_u t}c^{\frac{1}{4}}\widetilde{Z}_{n(t_0)}(a_u)+\psi_t''(\rho)c^{\frac{1}{2}}\widetilde{Z}_{n(t_0)}(a_u)^2 =  \psi_t\left(a_u+c^{\frac{1}{4}}\widetilde{Z}_{n(t_0)}(a_u)\right)
\]
with $a_u\leq \rho \leq a_u+c^{\frac{1}{4}}\widetilde{Z}_{n(t_0)}(a_u)$.
By differentiating the expression in \eqref{derivpsi}
we get 
\[
\psi_t''(x) = \psi_t'(x) \int_0^t b''(\psi_s(x))\psi'_s(x) ds.
\]
Since $b'$ is decreasing, $b''(x) \leq 0$ and $b'(x) \leq \lambda_u$ for $x \in  [a_u, x_+]$. Furthermore if $t \leq T_1$ then $\psi_t(\rho) \in [a_u, x_+]$ and hence $\psi_t'(\rho) \leq e^{\lambda_u t}$. Therefore
\[
0 \leq -\psi_t''(\rho) \leq \|b''\|_{\infty} e^{\lambda_u t} \int_0^t  e^{\lambda_u s} ds \leq \frac{\|b''\|_{\infty}}{\lambda_u} e^{2 \lambda_u t}
\] 
and so
\[
a_u+e^{\lambda_u t}c^{\frac{1}{4}}\widetilde{Z}_{n(t_0)}(a_u)- \frac{\|b''\|_{\infty}}{\lambda_u} e^{2 \lambda_u t} c^{\frac{1}{2}}\widetilde{Z}_{n(t_0)}(a_u)^2 \leq \psi_t\left(a_u+c^{\frac{1}{4}}\widetilde{Z}_{n(t_0)}(a_u)\right).
\]
Now take $t=\frac{1}{4\lambda_u}\log(c^{-1})+\frac{1}{\lambda_u}\log\left(\frac{2(x_+-a_u)}{\widetilde{Z}_{n(t_0)}(a_u)}\right)$. Then the left-hand side is equal to
\[
a_u + 2(x_+ - a_u) - \frac{4\|b''\|_{\infty}}{\lambda_u}(x_+-a_u)^2 = x_+ + (x_+ - a_u) \left ( 1 - \frac{4\|b''\|_{\infty}}{\lambda_u}(x_+-a_u) \right ) > x_+,
\]
where the final inequality used the assumptions made on $x_\pm$ at the beginning of Section \ref{subsec:anal}.
It follows that
\[
T_1 \leq \frac{1}{4\lambda_u}\log(c^{-1})+\frac{1}{\lambda_u}\log\left(\frac{2(x_+-a_u)}{\widetilde{Z}_{n(t_0)}(a_u)}\right) \leq \frac{1}{4\lambda_u}\log(c^{-1})+\frac{1}{\lambda_u}\log\left(\frac{2(x_+-a_u)}{C_\epsilon}\right)
\]
as required.

We now do the case when $b'(x)$ is increasing. Using a similar argument to that used to establish \eqref{eq:tayexp}, we have
\[
a_u+e^{\lambda_u t} c^{\frac{1}{4}}\widetilde{Z}_{n(t_0)}(a_u) \leq \psi_t(a_u+c^{\frac{1}{4}}\widetilde{Z}_{n(t_0)}(a_u))\leq a_u+e^{I(0,t)} c^{\frac{1}{4}}\widetilde{Z}_{n(t_0)}(a_u).
\]
From the left-hand side
\begin{equation}
\label{eq:T1upbi}
T_1 \leq \frac{1}{\lambda_u}\log\left(\frac{x_+-a_u}{c^{\frac{1}{4}}\widetilde{Z}_{n(t_0)}(a_u)}\right) \leq \frac{1}{4 \lambda_u} \log(c^{-1}) + \log((x_+-a_u)/C_\epsilon) < c^{-1/2}
\end{equation}
for $c$ sufficiently small and hence we can discard the $c^{-\frac{1}{2}}$ upper bound in the definition of $T_1^*$ in \eqref{Tstarbound}. 
To show that $T_1=T_1^*$, consider the Taylor expansion
\begin{align*}
a_u &= \psi_t \left (a_u + c^{\frac{1}{4}}\widetilde{Z}_{n(t_0)}(a_u) - c^{\frac{1}{4}}\widetilde{Z}_{n(t_0)}(a_u) \right ) \\
&= \psi_t \left (a_u + c^{\frac{1}{4}}\widetilde{Z}_{n(t_0)}(a_u) \right ) - c^{\frac{1}{4}}\widetilde{Z}_{n(t_0)}(a_u) \psi_t'\left (a_u + c^{\frac{1}{4}}\widetilde{Z}_{n(t_0)}(a_u) \right ) + \frac{1}{2} \psi_t''(\rho)c^{\frac{1}{2}}\widetilde{Z}_{n(t_0)}(a_u)^2
\end{align*}
for some $\rho \in [a_u, a_u + c^{\frac{1}{4}}\widetilde{Z}_{n(t_0)}(a_u)]$. As above,
\[
\psi_t''(x) = \psi_t'(x) \int_0^t b''(\psi_s(x))\psi'_s(x) ds.
\]
Hence, using that $b''(x) \geq 0$ and $b'(x) \geq \lambda_u$ for $x \in  [a_u, x_+]$,
\[
0 \leq \psi_t''(x) \leq \frac{\|b''\|_{\infty}}{\lambda_u} \psi_t'(x) \int_0^t b'(\psi_s(x))\psi'_s(x) ds = \frac{\|b''\|_{\infty}}{\lambda_u} \psi_t'(x)(\psi_t'(x) - 1) \leq \frac{\|b''\|_{\infty}}{\lambda_u} \psi_t'(x)^2.
\] 
Using that $\psi'_t(x)$ is increasing in $x$ and that 
\[
\psi_t'(a_u + c^{\frac{1}{4}}\widetilde{Z}_{n(t_0)}(a_u)) = e^{I(0,t)}, 
\]
we have
\[
0 \leq \psi_t''(\rho) \leq \frac{\|b''\|_{\infty}}{\lambda_u} e^{2I(0,t)}.
\]
Rearranging the Taylor expansion gives
\begin{align*}
x_+ &\geq \psi_t(a_u + c^{\frac{1}{4}}\widetilde{Z}_{n(t_0)}(a_u))  \\
&= a_u + c^{\frac{1}{4}}\widetilde{Z}_{n(t_0)}(a_u) \psi_t'(a_u + c^{\frac{1}{4}}\widetilde{Z}_{n(t_0)}(a_u)) - \frac{1}{2} \psi_t''(\rho)c^{\frac{1}{2}}\widetilde{Z}_{n(t_0)}(a_u)^2 \\
&\geq a_u + c^{\frac{1}{4}}\widetilde{Z}_{n(t_0)}(a_u) e^{I(0,t)} - \frac{\|b''\|_{\infty}}{2\lambda_u} e^{2I(0,t)} c^{\frac{1}{2}}\widetilde{Z}_{n(t_0)}(a_u)^2. 
\end{align*}
Now suppose that $T_1 \neq T_1^*$ so there exists some $t < T_1$ such that $e^{I(0,t)}>c^{-\frac{1}{4}}e^{\frac{I(0,t_0)}{8}}$. Then, using that $e^{\frac{I(0,t_0)}{8}} \to \infty$ as $c \to 0$ and then $t_0 \to \infty$, for $c$ sufficiently small and $t_0$ sufficiently large, there exists some (possibly different) $t < T_1$ such that 
\[
e^{I(0,t)} = 2 c^{-1/4} \frac{x_+-a_u}{Z_{n(t_0)}(a_u)}. 
\] 
But then
\begin{align*}
x_+ &\geq a_u + 2(x_+-a_u) - \frac{2 \|b'\|_{\infty}}{\lambda_u}(x_+-a_u)^2 \\
&= x_+ + (x_+-a_u) \left ( 1 - \frac{2 \|b'\|_{\infty}}{\lambda_u}(x_+-a_u) \right ) \\
&\geq x_+ + (x_+-a_u)/2 > x_+, 
\end{align*}
where the last line used the assumptions made on $x_\pm$ at the beginning of Section \ref{subsec:anal}. This gives the required contradiction and so $T_1=T_1^*$. Note that we have also proved the stronger statement
\[
\sup_{t \leq T_1} e^{I(0,t)} \leq 2 c^{-1/4} \frac{x_+-a_u}{Z_{n(t_0)}(a_u)}.
\]

It remains to show that, as $c\to 0$ and $t_0\to \infty$, $T_1^*$ is within a compact time of $\frac{1}{4\lambda_u}\log(c^{-1})$. The upper bound follows by equation \eqref{eq:T1upbi} so we just need to find the lower bound. 

By Taylor expanding around $a_u$,
\[
\psi_t\left(a_u+c^{\frac{1}{4}}\widetilde{Z}_{n(t_0)}(a_u)\right) \leq a_u+e^{\lambda_u t}c^{\frac{1}{4}}\widetilde{Z}_{n(t_0)}(a_u)+\psi_t''(\rho)c^{\frac{1}{2}}\widetilde{Z}_{n(t_0)}(a_u)^2
\]
with $a_u\leq \rho \leq a_u+c^{\frac{1}{4}}\widetilde{Z}_{n(t_0)}(a_u)$.
Exactly as above, if $t \leq T_1$ 
\[
0 \leq \psi_t''(\rho) \leq \frac{\|b''\|_{\infty}}{\lambda_u} e^{2I(0,t)} < \frac{4\|b''\|_{\infty}}{\lambda_u} c^{-1/2} \left ( \frac{x_+-a_u}{\widetilde{Z}_{n(t_0)}(a_u)} \right )^2
\]
and hence, using that $\widetilde{Z}_{n(t_0)}(a_u) < 1/C_{\epsilon}$
\[
\psi_t\left(a_u+c^{\frac{1}{4}}\widetilde{Z}_{n(t_0)}(a_u)\right) < a_u+e^{\lambda_u t}c^{\frac{1}{4}}C_{\epsilon}^{-1}+\frac{4\|b''\|_{\infty}}{\lambda_u}(x_+-a_u)^2.
\]
Now take $t=\frac{1}{4\lambda_u}\log(c^{-1})+\frac{1}{\lambda_u}\log\left(\frac{(x_+-a_u)C_\epsilon}{2}\right)$. Then the right-hand side is equal to
\[
a_u + \frac{1}{2}(x_+-a_u) + \frac{4\|b''\|_{\infty}}{\lambda_u}(x_+-a_u)^2 = x_+ - \frac{1}{2}(x_+-a_u)\left ( 1 - \frac{8\|b''\|_{\infty}}{\lambda_u}(x_+-a_u) \right ) < x_+.
\]
It follows that
\[
T_1 \geq \frac{1}{4\lambda_u}\log(c^{-1})+\frac{1}{\lambda_u}\log\left(\frac{(x_+-a_u)C_\epsilon}{2}\right)
\]
as required.

Therefore we have shown that
\[
\lim_{t_0 \to \infty} \lim_{c \to 0} \mathbb{P}(T_1 = T_1^*) \geq \lim_{t_0 \to \infty} \lim_{c \to 0} \mathbb{P}\left( C_{\epsilon}<|\widetilde{Z}_{n(t_0)}(a_u)|< \frac{1}{C_{\epsilon}}\right)>1-\epsilon.
\]
Since this holds for all $\epsilon > 0$, the first part of the theorem follows.

For the second part, set
\[
\widetilde{T_\epsilon} = \max \left \{ \frac{1}{\lambda_u} \left | \log\left(\frac{(x_+-a_u)C_\epsilon}{2}\right) \right |, \frac{1}{\lambda_u} \left | \log\left(\frac{2(x_+-a_u)}{C_\epsilon}\right) \right | \right \}.
\]
Then the analysis above shows that
$$\lim_{t_0\to \infty}\lim_{c\to 0}\mathbb{P}\left(\left|T_1^*-\frac{1}{4\lambda_u}\log (c^{-1})\right| \leq \widetilde{T_\epsilon}\right) \geq \lim_{t_0 \to \infty} \lim_{c \to 0} \mathbb{P}\left( C_{\epsilon}<|\widetilde{Z}_{n(t_0)}(a_u)|< \frac{1}{C_{\epsilon}}\right)>1-\epsilon
$$
as required.
\end{proof}

\section{Convergence of the whole trajectory}\label{sec:convtraj}
In this section we prove our main result that $$\sup_{t \in [0, \infty)}\left|X_{n(t)}(a_u)-\psi_{t}\left(a_u+c^{\frac{1}{4}}Z_{\infty}(a_u)\right)\right|\to 0$$
in probability as $c\to 0$. 

As $b(x)$ is periodic with period 1, $a_u-1$ and $a_u+1$ are also unstable fixed points. Since fixed points of real valued ODEs alternate between being stable and unstable, there exist stable fixed points $a_s^+$ and $a_s^-$ such that $a_u-1<a_s^-<a_u<a_s^+<a_u+1$ and we assume that $a_s^\pm$ are chosen so that there are no other fixed points in the interval $(a_s^-,a_s^+)$. 

The next two theorems extend the results from the previous sections to summarise the behaviour of the trajectory up to the end of the critical window, by which time the trajectory will have moved away from the unstable point towards one of the stable points. We finish by showing that that once the harmonic measure flow gets close enough to one of the stable fixed points, it remains near the stable trajectory for all time. 

\begin{theorem}\label{4.3a}
Let $T\geq 0$ be fixed. Then for any $\epsilon>0$
$$\lim_{t_0\to \infty}\lim_{c\to 0}\mathbb{P}\left(\sup_{t_0 \leq t \leq \frac{1}{4\lambda_u}\log(c^{-1})+T}|X_{n(t)}(a_u)-\psi_{t-t_0}\left(X_{k_0}\left(a_u\right)\right)|>\epsilon\right)=0.$$

\end{theorem}
\begin{proof}
We shall show that
$$\lim_{t_0\to \infty}\lim_{c\to 0}\mathbb{P}\left(\sup_{t_0 \leq t \leq T_1^*+T}|X_{n(t)}(a_u)-\psi_{t-t_0}\left(X_{k_0}\left(a_u\right)\right)|>\epsilon\right)=0.$$
The result then immediately follows by applying Theorem \ref{4.4a}. In fact,  only need to consider the case when $T_1^* < t \leq T_1^* + T$ since the case $t_0 \leq t \leq T_1^*$ follows directly from Theorem \ref{4.2}, using that $e^{-\frac{1}{2}I(0, t_0)}$ can be made arbitrarily small by taking $t_0$ large enough.

Recall the construction of $X_n$ from Section \ref{sec:harm} which corresponds to saying 
\[
X_n(x)=\frac{1}{2\pi i}\log (\Gamma_n(e^{2\pi i x}))
\]
where $\Gamma_n(x)=\phi_{n}^{-1}(x)=f_n^{-1}\circ \cdots \circ f_1^{-1}(x)$. For $n>k$, let
\[
\Gamma_{n,k}(x)=f_n^{-1}\circ \cdots \circ f_{k+1}^{-1}(x)
\]
so $\Gamma_{n,0}=\Gamma_n$. We can therefore construct in exactly the same way
\[
X_{n,k}(x)=\frac{1}{2\pi i}\log (\Gamma_{n,k}(e^{2\pi i x})).
\]
Since $\Gamma_{n,k}=\Gamma_{n,m}\circ \Gamma_{m,k}$ for $k<m<n$,
\begin{align*}
X_n(x)
&=X_{n,k}(X_k(x)).
\end{align*}
For $T_1^*\leq t\leq T_1^*+T$,
\begin{align*}
X_{n(t)}(a_u)= X_{n(t),n(T_1^*)}\left(X_{n(T_1^*)}(a_u)\right).
\end{align*}
Therefore,
\begin{align*}
|X_{n(t)}(a_u)-\psi_{t-t_0}\left(X_{k_0}\left(a_u\right)\right)|\leq&| X_{n(t),n(T_1^*)}\left(X_{n(T_1^*)}(a_u)\right)-\psi_{t-T_1^*}\left(X_{n(T_1^*)}(a_u)\right)|\\
&+|\psi_{t-T_1^*}\left(X_{n(T_1^*)}(a_u)\right)-\psi_{t-T_1^*}\left(\psi_{T_1^*-t_0}(X_{k_0}\left(a_u\right))\right)|
\end{align*}
We note that $X_{n,k}$ and $X_{n-k}$ are measurable with respect to different $\sigma$-algebras dependent on the choice of angles but are equal in distribution. Therefore, we can apply a version of Theorem \ref{3.6II} to show that for any $\epsilon>0$ and fixed $T\geq 0$,
$$\limsup_{c\to 0} \mathbb{P}\left(\sup_{T_1^*\leq t\leq T_1^*+T} | X_{n(t),n(T_1^*)}\left(X_{n(T_1^*)}(a_u)\right)-\psi_{t-T_1^*}\left(X_{n(T_1^*)}(a_u)\right)|>\epsilon\right)=0.$$
For the second term, 
\begin{align*}
&\sup_{T_1^*\leq t\leq T_1^*+T}|\psi_{t-T_1^*}\left(X_{n(T_1^*)}(a_u)\right)-\psi_{t-T_1^*}\left(\psi_{T_1^*-t_0}(X_{k_0}\left(a_u\right))\right)|\\
&\leq \sup_{T_1^*\leq t\leq T_1^*+T}\|\psi_{t-T_1^*}'\|_{\infty}|X_{n(T_1^*)}(a_u)-\psi_{T_1^*-t_0}(X_{k_0}\left(a_u\right))|\\
&\leq e^{\|b'\|_{\infty} T}|X_{n(T_1^*)}(a_u)-\psi_{T_1^*-t_0}(X_{k_0}\left(a_u\right))|
\end{align*}
where the last inequality follows by equation \eqref{psiinfbound}. Therefore, by Theorem \ref{4.2},
$$\mathbb{P}\left(\sup_{T_1^*\leq t\leq T_1^*+T}|\psi_{t-T_1^*}\left(X_{n(T_1^*)}(a_u)\right)-\psi_{t-T_1^*}\left(\psi_{T_1^*-t_0}(X_{k_0}\left(a_u\right))\right)|>e^{\|b'\|_{\infty} T-\frac{1}{2}I(0,t_0)}\right)\to 0$$
as $c\to 0$ and then $t_0\to \infty$. As $T\geq 0$ is fixed, for any $\epsilon>0$ we can choose $t_0$ large enough such that $0<e^{\|b'\|_{\infty} T}e^{-\frac{1}{2}I(0,t_0)}<\epsilon$, whence the result follows. 
\end{proof}

\begin{theorem}\label{4.5a}
Let $T\geq 0$ be fixed. Then
$$\sup_{0 \leq t \leq \frac{1}{4\lambda_u}\log(c^{-1})+T}\left|X_{n(t)}(a_u)-\psi_{t}\left(a_u+c^{\frac{1}{4}}Z_{\infty}(a_u)\right)\right| \to 0$$
in probability as $c \to 0$. 
\end{theorem}
\begin{proof}
By Theorem \ref{3.6II}, 
\[
\sup_{0 \leq t \leq t_0} |X_{n(t)}(a_u) - \psi_t(a_u)| \to 0
\]
in probability as $c \to 0$. But also, by the mean-value theorem,
\[
\sup_{0 \leq t \leq t_0} \left | \psi_t(a_u+c^{\frac{1}{4}}Z_{\infty}(a_u) - \psi_t(a_u) \right | \leq \sup_{0 \leq t \leq t_0} \sup_{a_s^- \leq x \leq a_s^+} \psi_t'(x) c^{\frac{1}{4}}|Z_{\infty}(a_u)| \to 0
\]
in probability as $c \to 0$. Hence
$$\sup_{0\leq t\leq t_0}\left|X_{n(t)}(a_u)-\psi_{t}\left(a_u+c^{\frac{1}{4}}Z_{\infty}(a_u)\right)\right| \to 0$$
in probability as $c \to 0$. It therefore suffices to consider the case when $t > t_0$.

We split
\begin{align*}
\left|X_{n(t)}(a_u)-\psi_{t}\left(a_u+c^{\frac{1}{4}}Z_{\infty}(a_u)\right)\right|\leq & \left|X_{n(t)}(a_u)-\psi_{t}\left(a_u+c^{\frac{1}{4}}\widetilde{Z}_{n(t_0)}(a_u)\right)\right|\\
&+\left|\psi_t\left(a_u+c^{\frac{1}{4}}\widetilde{Z}_{n(t_0)}(a_u)\right)-\psi_{t}\left(a_u+c^{\frac{1}{4}}Z_{\infty}(a_u)\right)\right|.
\end{align*}
The result above gives that
$$\sup_{t_0<t\leq \frac{1}{4\lambda_u}\log(c^{-1})+T}\left|X_{n(t)}(a_u)-\psi_{t}\left(a_u+c^{\frac{1}{4}}\widetilde{Z}_{n(t_0)}(a_u)\right)\right| \to 0$$
in probability as $c \to 0$ and then $t_0 \to \infty$. Hence, all that remains to show is that $$\sup_{t_0<t \leq \frac{1}{4\lambda_u}\log(c^{-1})+T} \left|\psi_t\left(a_u+c^{\frac{1}{4}}\widetilde{Z}_{n(t_0)}(a_u)\right)-\psi_{t}\left(a_u+c^{\frac{1}{4}}Z_{\infty}(a_u)\right)\right|$$ also converges to $0$ in probability. 

Let 
\[
T_2 =\inf\left\lbrace t\geq 0: \psi_{t}\left ( a_u + c^{\frac{1}{4}} Z_{\infty}(a_u) \right )\not \in [x_-,x_+]\right\rbrace.
\]
By exactly the same arguments as those used in the proof of Theorem \ref{4.4a},  for each $\epsilon > 0$ there exists some constant $\widetilde{T}_\epsilon>0$ such that 
$$\lim_{t_0\to \infty}\lim_{c\to 0}\mathbb{P}\left(\left|T_2-\frac{1}{4\lambda_u}\log (c^{-1})\right| > \widetilde{T}_\epsilon\right) < \epsilon.
$$
Set $$A = \left [ a_u + c^{1/4} \left ( \widetilde{Z}_{n(t_0)}(a_u) \wedge Z_{\infty}(a_u) \right ), a_u + c^{1/4} \left ( \widetilde{Z}_{n(t_0)}(a_u) \vee Z_{\infty}(a_u)\right ) \right ].$$ It also follows from the proof of Theorem \ref{4.4a} that for each $\epsilon > 0$ there exists some constant $\widetilde{C}_\epsilon>0$ such that
\[
\lim_{t_0\to \infty}\lim_{c\to 0}\mathbb{P}\left( \sup_{t \leq T_1 \wedge T_2} \sup_{x \in A } \psi_t'(x) > c^{-1/4}\widetilde{C}_\epsilon\right) < \epsilon.
\]
Hence, on the high probability event
\[
\left \{ \left|T_1 \wedge T_2-\frac{1}{4\lambda_u}\log (c^{-1})\right| \leq \widetilde{T}_\epsilon \right \} \cap \left \{ \sup_{t \leq T_1 \wedge T_2} \sup_{x \in A } \psi_t'(x) \leq c^{-1/4}\widetilde{C}_\epsilon \right \},
\]
if $x \in A$ and $T_1 \wedge T_2 < t < \frac{1}{4\lambda_u}\log(c^{-1})+T$,
\[
\psi_t'(x) = \psi_{T_1 \wedge T_2}'(x) \exp \left ( \int_{T_1 \wedge T_2}^t b'(\psi_s(x)) ds \right ) \leq c^{-1/4}\widetilde{C}_\epsilon e^{\|b'\|_{\infty} (T+\widetilde{T}_\epsilon)}.
\]
By the mean value theorem,
\begin{align*}
\left|\psi_t\left(a_u+c^{\frac{1}{4}}\widetilde{Z}_{n(t_0)}(a_u)\right)-\psi_{t}\left(a_u+c^{\frac{1}{4}}Z_{\infty}(a_u)\right)\right|&\leq \sup_{ x \in A} \psi_t '(x) c^{\frac{1}{4}} \left|\widetilde{Z}_{n(t_0)}(a_u)- Z_{\infty}(a_u)\right|.
\end{align*}
Therefore
\begin{align*}
&\sup_{t \leq \frac{1}{4\lambda_u}\log(c^{-1})+T} \left|\psi_t\left(a_u+c^{\frac{1}{4}}\widetilde{Z}_{n(t_0)}(a_u)\right)-\psi_{t}\left(a_u+c^{\frac{1}{4}}Z_{\infty}(a_u)\right)\right|\\
&\leq \widetilde{C}_\epsilon e^{\|b'\|_{\infty} (T+\widetilde{T}_\epsilon)} \left|\widetilde{Z}_{n(t_0)}(a_u)- Z_{\infty}(a_u)\right|\\
&\leq \widetilde{C}_\epsilon e^{\|b'\|_{\infty} (T+\widetilde{T}_\epsilon)} \left|\widetilde{Z}_{n(t_0)}(a_u)- Z_{t_0}(a_u)\right|+\widetilde{C}_\epsilon e^{\|b'\|_{\infty} (T+\widetilde{T}_\epsilon)}\left|Z_{t_0}(a_u)- Z_{\infty}(a_u)\right|.
\end{align*}
By our assumptions at the start of Section \ref{sec:crit}, $\widetilde{Z}_{n(t_0)}(a_u)\to Z_{t_0}(a_u)$ in probability as $c\to 0$ and $Z_{t_0}(a_u)\to Z_{\infty}(a_u)$ in probability as $t_0\to \infty$. As a result, 
$$\sup_{0<t<\frac{1}{4\lambda_u}\log(c^{-1})+T}\left|\psi_t\left(a_u+c^{\frac{1}{4}}\widetilde{Z}_{n(t_0)}(a_u)\right)-\psi_{t}\left(a_u+c^{\frac{1}{4}}Z_{\infty}(a_u)\right)\right|\to 0$$
in probability as $c \to 0$ and then $t_0 \to 0$. 
\end{proof}
Finally we prove that once the harmonic measure flow gets close enough to the stable point it remains close to the stable trajectory for all time. 
\begin{theorem}\label{4.7a}
Suppose $a_s$ is a stable fixed point of $\psi_t(x)$ with $b'(a_s) = \lambda_s < 0$. Let $x$ be chosen sufficiently close to $a_s$ that $3 \lambda_s/2 < b'(y)< \lambda_s/2$ for all $y$ in the interval between $x$ and $a_s$. Then for any $\epsilon>0$,
$$\lim_{c\to 0}\mathbb{P}\left(\sup_{0\leq t< \infty} \left|X_{n(t)}(x)-\psi_{t}\left(x\right)\right|>\epsilon\right)= 0.$$
\end{theorem}
\begin{proof}
The proof uses very similar methods to those in Section \ref{sec:crit}. As a result, just an outline of the main steps is provided.

Set
$$\widehat{I}(t_1,t_2):=\int_{t_1}^{t_2} b'(\psi_{s}(x)ds.$$
By the stability of $a_s$ and the assumptions on $x$, 
 $b'(\psi_t(x))< \lambda_s/2$ for all $t \geq 0$, so $\widehat{I}(t_1,t_2) < \lambda_s(t_2-t_1)/2 < 0$. 
Let
$$h(t,y):= e^{-\widehat{I}(0,t)}\left(y-\psi_{t}(x)\right).$$
Write
\begin{align}\label{hequation}
h(nc,X_{n}(x))&=\widehat{M}(a_s,n)+\widehat{L}(a_s,n)+\sum_{i=0}^{n-1}\left(cb'(\psi_{ic-}(x))-\widehat{I}((i-1)c,ic)\right)h(ic,X_{i}(x))).
\end{align}
where $\widehat{M}(a_s,n)=\sum_{i=0}^{n-1}e^{-\widehat{I}(0,ic)} Y_{i+1}(x)$ is a martingale term and $\widehat{L}(a_s,n)$ is a reminder term, analogous to those in \eqref{gequation}.
 
Define the stopping time
$$\widehat{T}_0=\inf_{r\geq 0}\{r :\; \left|X_{n(r)}(x)-\psi_{r}\left(x\right)\right|>c^{\frac{1}{6}}\}. $$
By the same argument as Lemma \ref{Lbound}, 
$$\sup_{0\leq r\leq t \wedge \widehat{T}_0}|\widehat{L}(a_s, n(r))|\leq c^{\frac{1}{5}}e^{-\widehat{I}(0,t)}.$$
The difference in the upper bound from that in Lemma \ref{Lbound} results from the change of sign of $b'$ near the stable point, which means $-I(0,r)$ is maximised by taking $r$ as large as possible. Using a similar method to that in Lemma \ref{Mbound},
\begin{align*}
\mathbb{P}\left( \sup_{0\leq r\leq t}|\widehat{M}(a_s,n(r))|> c^{\frac{1}{4}}\log(c^{-1})e^{-\widehat{I}(0,t)}\right) \to 0
\end{align*}
as $c \to 0$.
Then, by equation \eqref{hequation}, on the high probability event defined above, if $0\leq t< \widehat{T}_0$
\begin{align*}
|h(t,X_{n(t)}(x))|&\leq  2 c^{\frac{1}{5}} e^{-\widehat{I}(0,t)}+\sup_{0\leq r\leq t}\left|\sum_{i=0}^{n(r)-1}\left(cb'(\psi_{ic}(x))-\widehat{I}((i-1)c,ic)\right)h\left(ic,X_{i}(x)\right)\right|.
\end{align*}
As in the proof of Lemma \ref{4.2}
\[
\left|cb'(\psi_{ic}(x))-\widehat{I}((i-1)c,ic)\right| \leq c^2 \|b''\|_{\infty}\|b\|_{\infty}.
\]
Using that $|h(s, X_{n(s)})| \leq c^{1/6}e^{-\widehat{I}(0,s)}$ for all $s < \widehat{T}_0$, if $0\leq t< \widehat{T}_0$ then
\begin{align*}
\sup_{0\leq r\leq t}\left|\sum_{i=0}^{n(r)-1}\left(cb'(\psi_{ic}(x))-\widehat{I}((i-1)c,ic)\right)h(ic,X_{i}(x)))\right|&\leq c^{\frac{7}{6}}  \|b''\|_{\infty}\|b\|_{\infty} \sum_{i=0}^{n(t)-1}ce^{-\widehat{I}(0,ic)}\\
&\leq \frac{2c^{\frac{7}{6}}  \|b''\|_{\infty}\|b\|_{\infty} \int_0^t (-b'(\psi_s(x)))e^{-\widehat{I}(0,s)}ds}{|\lambda_s|}\\
&\leq \frac{2c^{\frac{7}{6}}  \|b''\|_{\infty}\|b\|_{\infty} e^{-\widehat{I}(0,t)}}{|\lambda_s|}.
\end{align*}
Therefore, if $0\leq t< \widehat{T}_0$, for $c$ sufficiently small, with high probability,
$$\left|X_{n(t)}(x)-\psi_{t}\left(x\right)\right|<  c^{\frac{1}{6}}$$
and thus with high probability, the stopping time $\widehat{T}_0$ never occurs. Hence,
 $$\lim_{c\to 0}\lim_{t\to \infty}\mathbb{P}\left(\sup_{0\leq r\leq t} \left|X_{n(r)}(x)-\psi_{r}\left(x\right)\right|>\epsilon\right)= 0.$$
Let $\Omega_t$ be the event, $$\Omega_t:=\left\lbrace \sup_{0\leq r\leq t} \left|X_{n(r)}(x)-\psi_{r}\left(x\right)\right|>\epsilon\right\rbrace$$ where $t$ is an integer. Since the events $\left\lbrace \Omega_t\right\rbrace_{t\geq 0}$ are increasing in $t$ it follows that $\lim_{t\to \infty}\mathbb{P}\left(\bigcup_{r=1}^t\Omega_r\right)=\mathbb{P}\left(\bigcup_{r=1}^{\infty}\Omega_r\right)$. Therefore
 $$\lim_{c\to 0}\mathbb{P}\left(\sup_{0\leq t<\infty} \left|X_{n(t)}(x)-\psi_{t}\left(x\right)\right|>\epsilon\right)= 0.$$
\end{proof}
Theorem \ref{4.5a} shows that when $0<t<\frac{1}{4\lambda_u}\log(c^{-1})+T$ the harmonic measure started at the unstable point $X_{n(t)}(a_u)$ moves a macroscopic distance from $a_u$ towards either $a_s^-$ or $a_s^+$. Once at this macroscopic distance the process will remain close to the trajectory started at $\psi_{t}\left(a_u+c^{\frac{1}{4}}Z_{\infty}(a_u)\right)$ which will converge towards the stable point $a_s^{\pm}$. However, by Theorem \ref{4.7a} once the process gets close to the stable point it will remain close to the ODE trajectory starting from that point for all time. Therefore, we can deduce the following corollary.
\begin{corollary}\label{corol4.8}
For all $\epsilon > 0$
$$\lim_{c\to 0}\mathbb{P}\left( \sup_{0<t<\infty}\left|X_{n(t)}(a_u)-\psi_{t}\left(a_u+c^{\frac{1}{4}}Z_{\infty}(a_u)\right)\right| > \epsilon \right ) = 0,$$
where $ Z_{\infty}(a_u)$ is a Gaussian with mean $0$ and variance given by $\frac{\rho_0h_{\nu}(a_u)}{2 \lambda_u}$.
\end{corollary}

\bibliographystyle{plain}

\bibliography{bibliography}

\end{document}